\documentclass[journal]{IEEEtran}
\usepackage{amssymb}
\usepackage{amsthm}
\usepackage[pdftex]{graphicx}
\usepackage{tikz}
\usepackage{framed}
\usepackage{bm}
\usepackage{cite}
\usepackage{comment}
\usepackage{algorithmic}
\usepackage{algorithm}
\usepackage{enumerate}
\usepackage{enumitem}
\usepackage{amsmath}
\usepackage{amsfonts}
\usepackage{color}
\usepackage{cases}
\usepackage{empheq}
\usepackage{matlab-prettifier}
\newtheorem{theorem}{Theorem}
\newtheorem{assumption}{Assumption}
\newtheorem{definition}{Definition}
\newtheorem{remark}{Remark}
\newtheorem{lemma}{Lemma}

\newtheorem{problem}{Problem}

\usepackage{blkarray}

\allowdisplaybreaks 

\usepackage{blkarray}

\usepackage{latexsym}

\usepackage{booktabs}
\usepackage{makecell}

\usepackage{cite}
\usepackage{amsmath}
  \usepackage[caption=false,font=normalsize,labelfont=sf,textfont=sf]{subfig}

\usepackage{stfloats}
\usepackage{url}

\begin{document}
\title{Model-Free Design and Analysis of 2DOF PI Controller for Noisy LTI Systems } 

\author{Taiga Kiyota and Kazuhiro Sato\thanks{T. Kiyota and K. Sato are with the Department of Mathematical Informatics, Graduate School of Information Science and Technology, The University of Tokyo, Tokyo 113-8656, Japan, email: kiyota-taiga850@g.ecc.u-tokyo.ac.jp (T. Kiyota), kazuhiro@mist.i.u-tokyo.ac.jp (K. Sato) }}
\maketitle
\thispagestyle{empty}
\pagestyle{empty}

\begin{abstract}
Set-point tracking for systems with unknown model parameters is a fundamental problem in control, and two-degree-of-freedom (2DOF) Proportional-Integral (PI) controllers---consisting of a feedforward controller and PI controller---are widely employed for this task.
In this paper, we propose a model-free design of 2DOF PI controllers, establish its theoretical properties, and compare them with a model-based method from both theoretical and numerical perspectives.
For the feedforward design, we extend an existing model-free algorithm to systems subject to Gaussian process and measurement noises. 
We derive a nonasymptotic lower bound on the required control input/output data length and characterize the resulting estimation error. 
For PI gain tuning, we formulate a constrained optimization problem and establish sample complexity of a zeroth-order optimization method.
Moreover, we quantify how inaccuracies in the feedforward design propagate to
the performance of the PI controller, highlighting an interaction that has not been examined in prior work.
We further provide a theoretical comparison between the proposed method and the model-based method.
In particular, for PI gain tuning, the proposed method is computationally more efficient by avoiding explicit gradient computations.
Numerical experiments demonstrate that the 2DOF PI controller designed by the proposed method exhibits better control performance than the model-based method.
\end{abstract}


\IEEEpeerreviewmaketitle

\section{Introduction}\label{sec:intro}
The set-point tracking problem refers to the control task of driving a system's output to a desired constant set-point. 
This problem is fundamental in control theory and arises in a wide range of practical applications,
including motion control \cite{XueWenchaoEtal2017}, attitude control of aircraft \cite{Xia_et_al_2011}, and process control \cite{Zhou_etal_2018}. 
To address this problem, a two-degree-of-freedom (2DOF) PI controller, consisting of a feedforward controller and a PI controller, is commonly employed.
2DOF PI controllers are extensively utilized in industrial applications such as
temperature control \cite{Jain_etal_2019}, 
permanent magnet synchronous machines \cite{Hussain_2021,Kai_etal_2020}, 
and power systems\cite{Huang_etal_2015,Yongqing_etal_2019}.\par
Conventionally, controller design methodologies rely on model information of the target system.
Design of the feedforward term \cite{Araki_Taguchi03, VITECKOVA_2010, Youla_et_al_1985, Wang_Yan_et_al_2018}
and the PI gain \cite{He_Wang_2006, Boyd_et_al_2016} of 
a 2DOF PI controller often depends on model information. 
In practice, however, the complete model information is not readily available.
One approach to address this limitation is to estimate a model through system identification, and then design the controller, referred to as a model-based method. However,
this procedure can be costly and difficult when dealing with complex systems and in the presence of process and measurement noise. 
These challenges have motivated the development of model-free control approaches,
in which controllers are designed directly from input/output data without model information.\par

In this paper, we focus on the model-free design of a 2DOF PI controller for the set-point tracking problem in MIMO linear time-invariant (LTI) systems with additive Gaussian noise.
For unknown MIMO LTI systems, model-free methods have been proposed separately for
feedforward design \cite{Davison_1976, Roszak_Davison_2008} and for PI gain tuning \cite{Karimi_2017, Schuchert_et_al_2024}. 
In the context of feedforward design, previous works \cite{Davison_1976,Roszak_Davison_2008} do not take into account situations where stochastic process and measurement noises are added to the system.
For PI gain tuning, the previous works \cite{Karimi_2017, Schuchert_et_al_2024} do not establish any theoretical guarantees on sample complexity.
Moreover, in the model-free setting, the controller performance degrades due to the lack of access to system parameters. 
Nonetheless, none of the above studies consider how performance degradation in one component of 2DOF controller (feedforward or PI controller) affects the design of the other component.
To the best of our knowledge, few studies 
($\mathrm{i}$) have provided a theoretical analysis of model-free 2DOF PI controller design with respect to sample complexity, 
($\mathrm{ii}$) have established how inaccuracies in one component affect the performance of the other, and,
($\mathrm{iii}$) have conducted a comparative analysis against model-based methods.
\par

The main contributions of this paper are summarized below:
\begin{itemize} 
    \item We develop a model-free 2DOF PI controller design method for MIMO LTI systems with additive Gaussian noise,
    and provide its theoretical analysis. 
    The proposed method integrates ($\mathrm{i}$) the feedforward design and ($\mathrm{ii}$) the PI gain tuning.
    \item For feedforward design, we establish theoretical results for sample complexity.
    While our algorithm builds upon the classical approaches \cite{Davison_1976, Roszak_Davison_2008}, 
    we characterize the bounds on the time horizon of the input/output data and the estimation error of the feedforward.
    The resulting feedforward controller is subsequently used to generate the input/output data in the PI gain tuning algorithm.
    \item For PI gain tuning, we develop theoretical analyses of the sample complexity and of the performance degradation caused by feedforward estimation errors.
    Building on a model-free policy gradient method\cite{Convergence_Complexity_IEEE, Fazal_18, Takakura_Sato_24}, we
    establish the theoretical guarantees on the sample complexity of the input/output data.
    Moreover, we present that the performance degradation of the PI controller scales as $O(\gamma_u^2)$, where $\gamma_u$ is the estimation error of the feedforward.
     \item We theoretically and numerically compare the proposed method with the model-based method. 
    For feedforward design, the advantage of the proposed method over the model-based method depends on the performance level required.
    When a moderate level of steady-state accuracy is sufficient, the proposed method requires a smaller number of samples.
    For PI gain tuning, although the proposed method requires a larger number of samples than those needed for system identification in the model-based method,
    its computational burden is much smaller via avoiding explicit gradient computations.
    Furthermore, numerical experiments demonstrate that, given the same length of samples, the proposed method outperforms the model-based method.
    
\end{itemize}\par
The remainder of this paper is structured as follows. 
In Section \ref{sec:pre}, we provide an overview of 2DOF PI controllers.
In Section \ref{sec:problem_formulation}, we formulate the problems of designing the feedforward controller and tuning the PI gain. 
In Section \ref{sec:ProposedMethod} presents proposed methods along with their theoretical analyses.
Section~\ref{sec:Compare_IndirectApproach} is devoted to the theoretical comparison with the model-based method.
In Section \ref{sec:numerical_experiment}, we conduct numerical experiments to  compare with the model-based method in terms of control performance.
Section \ref{sec:concluding_remarks} is devoted to the conclusion.\par
{\it Notation}:  
$I$ and $O$ denote the identity matrix and the zero matrix, respectively.
For a vector $v\in \mathbb{R}^n$, $\|v \|$ denotes the standard Euclidean norm of $v$, and $v^\top$ denotes the transpose of $v$. 
For a matrix $A\in \mathbb{R}^{n\times m}$, $A^\top$ denotes the transpose of $A$, and $\|A\|_{\rm{F}}$ and $\|A\|_2$ denote the Frobenius norm and the spectral norm of $A$, respectively,
Let us denote the $i$-th singular of $A$  by $\sigma_{i}(A)$, indexed as $\sigma_{1}(A)\geq \cdots \geq \sigma_{n}(A)$, and the minimum singular value of $A$ by $\sigma_{\min}(A)$. 
The eigenvalues of $A$ with the maximum and minimum real part are denoted by $\lambda_{\max}(A)$ and $\lambda_{\min}(A)$, respectively.
For a full row rank matrix $A\in \mathbb{R}^{n\times m}$, $A^+$ denotes $A^\top(AA^\top)^{-1}$.
For a symmetric matrix $A\in \mathbb{R}^{n\times n}$, $A\succ O$ and $A\succeq O$ means $A$ is positive definite and positive semi-definite, respectively.
For matrices $A,\,B\in \mathbb{R}^{n\times n}$, $A\otimes B$ denotes the Kronecker product of $A$ and $B$.
$\mathrm{proj}_\Omega :\mathbb{R}^{n\times m} \to \Omega $ denotes the orthogonal projection onto a closed convex set $\Omega$.
For a random variable $X$, let us denote the sub-exponential norm of $X$ by $\|X\|_{\psi_1}$.

\section{Preliminaries}\label{sec:pre}

In this section, we summarize the properties of a 2DOF PI controller for a linear time-invariant system:
\begin{equation}\label{eq:dynamics}
\begin{split}
    \dot{x}(t)=Ax(t)+Bu(t){+w(t) },\ y(t)=C x(t){+v(t) },\  x(0)\sim \mathcal{D},
    \end{split}
\end{equation}
where $x(t)\in \mathbb{R}^n$ is the state, $u(t) \in \mathbb{R}^m$ is the control input,
$y(t)\in \mathbb{R}^p$ is the output with $p\leq m$, 
$w(t)$ and $v(t)$ are zero-mean Gaussian noises with covariance $W\succeq O$ and $V\succ O$, respectively,
and $\mathcal{D}$ is a probability distribution over $\mathbb{R}^n$. 
A 2DOF PI controller is given by
\begin{equation}\label{eq:2DoF}
    \begin{split}
       { u(t)=K_P e(t) + K_I z(t) +u_0,} 
    \end{split}
\end{equation}
where $e(t)=y^\star -y(t)$ is the tracking error, $y^\star \in \mathbb{R}^p$ is the desired output set-point,
$z(t)$ is the integrator of the tracking error, satisfying $\dot{z}(t)=e(t)$, and $u_0\in \mathbb{R}^p$ is the constant feedforward input.
Denote an equilibrium of the system by $(x^\star,\,u^\star)$, which satisfies
\begin{equation}\label{eq:equilibrium}
    \begin{split}
        0&=Ax^\star + Bu^\star, \quad y^\star = Cx^\star.
    \end{split}
\end{equation}
To guarantee the existence of $(x^\star,\,u^\star)$ satisfying (\ref{eq:equilibrium}) for any $y^\star \in\mathbb{R}^p$, we impose the following assumption.
\begin{assumption}\label{assump:dynamics}
The matrix $\begin{pmatrix}
    A & B\\
    C & O
\end{pmatrix}$ is of full row rank.
\end{assumption}
\noindent Note that this assumption
is automatically satisfied if $A$ is invertible, $B$ is of full column rank, and $C$ is of full row rank.
\par
The dynamics of $e_x(t)\coloneqq x(t)-x^\star$ is governed by
\begin{equation}\label{eq:error_dynamics}
    \dot{e}_x(t)=Ae_x(t) + B(u(t)-u^\star).
\end{equation}
Combining the above equation, (\ref{eq:2DoF}) and $e(t)= - C e_x(t)$,
the closed-loop error dynamics of the augmented system can be written as
\begin{equation}\label{eq:dynamics_PI}
    \begin{split}
        \begin{pmatrix} \dot{e_x}(t)\\ \dot{z} (t)\end{pmatrix} &=
        {(\bar{A}-\bar{B}K\bar{C})}
        \begin{pmatrix} e_x(t)\\ z(t) \end{pmatrix}+\bar{B} (u_0-u^\star)\\
        &\quad { + \begin{pmatrix}w(t)+B K_P v(t)\\ v(t)\end{pmatrix} },\\
        e(t)&=\begin{pmatrix}
           -C & O
       \end{pmatrix}  \begin{pmatrix} e_x(t)\\ z(t) \end{pmatrix} 
       { +v(t)},
    \end{split}
\end{equation}
where ${ \bar{A} }= \begin{pmatrix} A& O\\ -C&O \end{pmatrix}, { \bar{B} }= \begin{pmatrix}B \\O \end{pmatrix},
    { \bar{C} }=\begin{pmatrix}C &O \\ O&- I \end{pmatrix},$
and 
{ $K=\begin{pmatrix}K_P & K_I  \end{pmatrix}$.}
 { The mean and the variance of the augmented system} can be expressed in closed form as
\begin{equation}\label{eq:mean_var_augdynamics}
    \begin{split}
        &{
        \mathbb{E}\left[\begin{pmatrix} e_x(t)\\ z(t)\end{pmatrix} \right]=
        \exp(\bar{A}_K t)
        \left( \mathbb{E}\left[\begin{pmatrix} e_x(0)\\ z(0)\end{pmatrix}\right] +\bar{A}_K^{-1}\bar{B}(u_0-u^\star) \right)
        }\\
        &\quad\quad\quad\quad\quad\quad { -\bar{A}_K^{-1}\bar{B}(u_0-u^\star) },\\
        &{
        \text{Var}\left( \begin{pmatrix} e_x(t)\\ z(t)\end{pmatrix} \right)
    = \exp{(\bar{A}_K t)} \bar{\Sigma}_0 \exp{(\bar{A}_K^\top t)}
    }\\
    &\quad\quad\quad\quad\quad\quad {+\int_0^t  \exp{(\bar{A}_K t)}
    \tilde{W}_K
   \exp{(\bar{A}_K^\top t)} d\tau
        }
    \end{split}
\end{equation}
where $\bar{A}_K\coloneqq {\bar{A}-\bar{B}K\bar{C} }$,
$\Bar{\Sigma}_0=\begin{pmatrix}\Sigma_0 & O \\ O & z(0)z(0)^\top\end{pmatrix},\,
\Sigma_0= \text{Var}_{x(0)}\left( e_x(0) \right)$
and $\tilde{W}_K\coloneqq\begin{pmatrix}W + B K_P V (B K_P)^\top & B K_PV \\ (B K_P V)^\top& V\end{pmatrix} $ .

\section{Problem formulation} \label{sec:problem_formulation}
In this section, we formulate the problems of designing a 2DOF PI controller for a MIMO LTI system (\ref{eq:dynamics}) under the following assumption.
\begin{assumption}\label{assump:model-free}
    The matrices $A,\,B,\,C {,\,W,\,V }$, and the distribution $\mathcal{D}$ of system (\ref{eq:dynamics}) are unknown.
\end{assumption}
\subsection{Feedforward Design}
Designing an appropriate feedforward input is essential for improving the transient performance of a 2DOF PI controller. 
A common approach is setting the feedforward input $u_0$ to the equilibrium input $u^\star$, which is obtained by solving (\ref{eq:equilibrium}).
However, solving (\ref{eq:equilibrium}) requires the parameters of the system matrices $A,\,B,\,C$, which are unavailable in the model-free setting. To address this issue, we consider the following problem.
\begin{problem}\label{problem:feedforward}
    Under Assumptions~\ref{assump:dynamics} and \ref{assump:model-free},
    design a feedforward input $\hat{u}_0$ that approximates the equilibrium input $u^\star$.
 \end{problem}

\subsection{PI Gain Tuning}
To design a 2DOF PI controller that achieves satisfactory performance, we formulate an optimization problem for tuning the PI gain $ K\coloneq \begin{pmatrix}K_P & K_I\end{pmatrix}\in \mathbb{R}^{m \times 2p} $.
To quantify performance, we employ the infinite-horizon average cost
$f(K) \coloneqq\lim_{T\to \infty}\frac{1}{T}\mathbb{E}\left[ \int_0^T e(t)^\top Q_1 e(t)+ z(t)^\top Q_2 z(t) dt\right]$, where $Q_1,\,Q_2\succeq O$.
Quadratic cost functions are widely used to evaluate transient performance \cite{Zhou_etal_1996_RobustOptimalControl,Astrom_2012,Hubert_Raphael_1972},
and the expectation averages out the dependence on the initial state $x(0)$ and stochastic disturbances $w,\,v$ \cite{Astrom_2012,Hubert_Raphael_1972,Athans_1971}.
We define $Q^\prime = \begin{pmatrix}
                C^\top Q_1 C & O\\ O& Q_2
            \end{pmatrix}$,
and then the cost function becomes $
        f(K)
        = \mathrm{tr}(Q^\prime X+Q_1V)
        $, 
where $X$ is the solution to the Lyapunov equation
$\bar{A}_K X+X\bar{A}_K^\top +\tilde{W}_K=O$ \cite[Lemma 2.2]{Zheng_etal_LQG_2021_arXiv}.
The existence and uniqueness of $X$ is guaranteed because $\bar{A}_K$ is Hurwitz and 
$\tilde{W}_K$ is positive definite \cite{Bernstein_2009}.
To ensure that the finiteness of $f(K)$ coincides with the closed-loop stability, we impose the following assumption:
\begin{assumption}\label{assump:detectability}
   The pair $(\bar{A}_K,{Q^\prime}^{1/2})$ is detectable.
\end{assumption}
\noindent The above assumption is satisfied when $(A,C)$ is detectable and $Q_1,\,Q_2\succ O$.
Furthermore, we impose constraints on the PI gain $K$:
$K$ is contained in a bounded closed convex set $\Omega \subset \mathbb{R}^{m \times 2p} $ that specifies the structural properties of $K$ or bounds the norm to prevent excessive input energy.
If the equilibrium input $u^\star$, which satisfies (\ref{eq:equilibrium}), is available,
we can set the feedforward input $u_0=u^\star$.\par
Based on the above discussion, we formulate the following optimization problem:
 \begin{problem}\label{problem:gain_optimization}
 Under Assumptions~\ref{assump:model-free}, \ref{assump:detectability}, solve the following constrained optimization problem:
     \begin{equation*}
    \begin{split}
        \min_{K \in \Omega} \quad &{f(K)= \mathrm{tr}(Q^\prime X+Q_1V)}  \\
        {\text{s.t.} } \quad &
        { \bar{A}_K X+X\bar{A}_K^\top +\tilde{W}_K=O}
        ,\quad { K\in \mathcal{S},}
    \end{split}
\end{equation*} 
where $\mathcal{S}\coloneqq \{K\in \mathbb{R}^{m\times 2p}\,|\, \bar{A}_K \text{is Hurwitz}\}$ is the set of stabilizing PI gains, and $\Omega$ is a bounded, closed, convex set.
 \end{problem}
\noindent Note that the cost function $f(K)$ can be nonconvex, and an example is provided in Appendix~\ref{appendix:Non-convexity}.
 
 \begin{remark}
     Constrained optimization of feedback gains has been studied in various contexts.
     Talebi and Mesbahi \cite{Talebi_etal_2024} proposed a policy optimization method for the linear-quadratic regulator 
    (LQR) problem with linear constraints;
     however, their approach is restricted to linear constraints, and they left the extension to nonlinear constraints for future work.
     Safety control formulations, such as cost-constrained \cite{Zhao_Keyou_2024} or risk-constrained LQR \cite{Tsiamis_etal_2020,Zhao_Keyou_2021}, impose cost function constraints, whose parameters can be difficult to tune.
     On the other hand, the constraint in our formulation is that the
     PI gain $K$ lies in a bounded closed convex set $\Omega$.
     Such a formulation provides interpretable and easily implementable constraints, such as
     norm bounds or structural properties, including positive definiteness for port-Hamiltonian systems \cite{LinearPortHamilton}.
 \end{remark}

\begin{remark}
Which approach---model-free or model-based---is advantageous depends on the control problem.
For the LQR problem, the model-based method achieves the sample complexity lower bound $\Omega(1/\epsilon)$ \cite{Mania_etal_2019}, where $\epsilon$ is the sub-optimality gap of the cost function.
In contrast, in the model-free approach, Malik et al. \cite{Malik_etal_2020} establish the upper bound $\tilde{O}(1/\epsilon)$ and Moghaddam et al. \cite{Moghaddam_etal_2025} also derive the same rate under milder assumptions.
For model reference control problems, model-free approach achieves better sample complexity when the model order is large \cite{Formentin_etal_2014}.
To the best of our knowledge, however, there has been little work comparing  model-free and model-based method for 2DOF PI controller in terms of sample and computational complexities.
\end{remark}
\par
\section{proposed method}\label{sec:ProposedMethod}
\subsection{Feedforward Design}\label{sec:proposed_feedforward}
In this subsection, we propose a method based on the approaches developed in \cite{Davison_1976, Roszak_Davison_2008}
to solve Problem \ref{problem:feedforward}.
Unlike the existing approaches \cite{Davison_1976, Roszak_Davison_2008},
our work addresses systems subject to Gaussian process noise and measurement noise. In addition, we provide
a lower bound on the required input/output data horizon and quantify the estimation error of the feedforward.
\par

\begin{algorithm}
 \begin{algorithmic}[1]
 \caption{Estimating $u^\star$ }\label{algo:DataMatrix}
            \REQUIRE $K_P^\prime,\,\tau_u>0$.
            \STATE Simulate the system (\ref{eq:dynamics})
            up to time $\tau_u$ with
            the input $u(t)=K_P^\prime e^0(t)$,  and obtain $ e^0(\tau_u)$.
            \FOR{$i=1,\ldots,\,m$}
            \STATE Set $u_0^i=\begin{pmatrix} 0&\ldots&0&1&0&\ldots& 0\end{pmatrix}\in\mathbb{R}^m$,
            the $i$-th standard basis vector of $\mathbb{R}^m$.
            \STATE Simulate (\ref{eq:dynamics}) up to time $\tau_u$
            with the input $u^i(t)=K_P^\prime e^i(t)+u_0^i$,  
             and obtain $ e^i(\tau_u)$.
            \ENDFOR
            \STATE Compute $e^{\prime i}(\tau_u) = e^i(\tau_u)-e^0(\tau_u)$ and 
            \begin{equation}\label{eq:DataMatrix_tau}
                E= \begin{pmatrix}
                    e^{\prime 1}(\tau_u) & \cdots &e^{\prime m}(\tau_u)
                \end{pmatrix}.
            \end{equation}
            \RETURN 
            \begin{equation}\label{eq:def_of_hatu}
                \hat{u}_0 =-E^+ e^0(\tau_u).
            \end{equation}
        \end{algorithmic}
\end{algorithm}

We propose Algorithm \ref{algo:DataMatrix} to design a feedforward input that approximates the equilibrium input $u^\star$, which satisfies (\ref{eq:equilibrium}).
In Step 4 of Algorithm \ref{algo:DataMatrix}, we use a 2DOF P controller $u(t)=K_P^\prime e(t)+u_0$ as the control input.
According to (\ref{eq:error_dynamics}) and $e(t)=-Ce_x(t)$, 
the closed-loop dynamics in Step~4 is given by
\begin{equation}\label{eq:stoc_loop_ff}
\begin{split}
    \dot{e}_x(t) =A_K e_x(t) +B(u_0 - u^\star)+BK_P v(t) + w(t),
\end{split}
\end{equation}
where $A_K \coloneqq A-B K_P^\prime C $, which is assumed to be Hurwitz.
Since the state error $e_x(t)$ follows a normal distribution whose mean is
\begin{equation}\label{eq:mean_Pcont}
\begin{split}
    &\mathbb{E} \left[  e_x(t) \right] 
    = \exp{(A_K t)} \left\{
    \mathbb{E} \left[  e_x(0) \right] +A_K^{-1} B (u_0-u^\star) \right\} \\
        &\quad\quad\quad\quad\quad-A_K^{-1} B(u_0-u^\star),
\end{split}
\end{equation}
the expectation of steady-state tracking error can be expressed as
\begin{equation}\label{eq:e_infty}
    {\mathbb{E} \left[  e(\infty) \right] }={ C A_K^{-1}B} (u_0-u^\star).
\end{equation}
 Let $e^0(t)$ denote the tracking error under $u(t)=K_P^\prime e^0(t)$.
 Similarly, let $e^i(t)$ denote the tracking error under  $u^i(t)=K_P^\prime e^i(t)+u_0^i$,
 where $u_0^i$ is the $i$-th standard basis vector of $\mathbb{R}^m$.
From (\ref{eq:e_infty}), we have
\begin{equation}\label{eq:e_0}
    {\mathbb{E} \left[  e^0(\infty) \right] }= -{ C A_K^{-1}B}u^\star,\quad
     {\mathbb{E} \left[  e^i(\infty) \right] }={ C A_K^{-1}B}(u_0^i-u^\star).
\end{equation}
 Define $e^{\prime i}(\cdot) = e^i(\cdot)-e^0(\cdot)$ and 
 \begin{equation}\label{eq:Estar}
     E^\star ={ \begin{pmatrix}
                    \mathbb{E} \left[  e^{\prime 1} \right] & \dots &\mathbb{E} \left[  e^{\prime m} \right]
                \end{pmatrix}}.
 \end{equation}
From (\ref{eq:e_0}) and $\begin{pmatrix}u_0^1 & \dots &u_0^m\end{pmatrix} =I$,
it follows that $E^\star =C A_K^{-1}B$.
The matrix $E^\star =C A_K^{-1}B$ is of full row rank when Assumption~\ref{assump:dynamics} is satisfied.
Therefore we can choose the equilibrium input $u^\star$ as
\begin{equation}\label{eq:ustar}
    u^\star = -{E^\star}^+ \mathbb{E}[e^0(\infty)].
\end{equation}
In Algorithm \ref{algo:DataMatrix}, $e^0(\tau_u)$, $E$, and $\hat{u}$ serve as approximations of $\mathbb{E}[e^0(\infty)]$, $E^\star$, and $u^\star$, respectively.
Equation~(\ref{eq:def_of_hatu}) has been employed in the algorithms of previous studies \cite{Davison_1976, Roszak_Davison_2008} for its simplicity.
Nevertheless, our contributions lie in addressing settings that were not considered in those works and establishing a theoretical result, as formalized in Theorem~\ref{thm:size_of_tau_u}.
The well-definedness of $\hat{u}_0$ is guaranteed by Lemma~\ref{lem:fullrank_E} in Appendix~\ref{appendix:well-definedness_FF}.
 
\begin{remark}\label{rem:Pgain_stabilizing}
    We assume that there exists a stabilizing P gain $K_P^\prime$ and that it is available.
    In Algorithm~\ref{algo:DataMatrix}, any stabilizing P gain can be used as $K_P^\prime$, and its selection is independent of the constraint set $\Omega$ in Problem~\ref{problem:gain_optimization}.
    Although the assumption that the stabilizing P gain can be obtained may appear strong, its construction is outside the scope of our work
    because the main objective of this paper is the theoretical analysis of Algorithm~\ref{algo:DataMatrix}.
    We note, however, that stabilizing P gains can be designed for certain classes of systems without knowledge of the system matrices $A,\,B,\,C$.
    For example, $K_P^\prime=O$ can be chosen for open-loop stable systems, and a positive definite gain can be chosen for port-Hamiltonian systems \cite{LinearPortHamilton}.
\end{remark}

The following theorem presents the required time horizon $\tau_u$ of the input/output data
 and the estimation error of the feedforward.
\begin{theorem}\label{thm:size_of_tau_u}
Suppose that $\sigma_p(C A_K^{-1} B)\geq 4\sqrt{2m\mathrm{tr}(C\Sigma C^\top) }$ holds,
where $\Sigma$ is the solution to (\ref{eq:Pcont_Lyapunov}).
For any $\epsilon_u \geq 0$, set
\begin{equation}\label{eq:bound_of_tau_u}
    \begin{split}
       \tau_u\geq \max\left\{ 2\|Z\|_2 \log \left( \max\left\{  \frac{M_1 }{ \epsilon_u} , M_3\right\} \right),\, \|Z\|_2 \log M_2 \right\},
    \end{split}
\end{equation}
where $M_1$, $M_2$, and $M_3$ are defined in (\ref{eq:def_M1}), (\ref{eq:def_M2}), and (\ref{eq:def_M3}), respectively, and $Z$ is the unique solution to the Lyapunov equation $A_K^\top Z +Z A_K+I=O$. 
  Then, for any $\delta_u>0$, we have 
  \begin{equation}\label{eq:bound_hatu_uncertainty}
        \begin{split}
            \| \hat{u}_0 -u^\star\|& \leq
            \epsilon_u + \bar{S}(\delta_u)
        \end{split}
    \end{equation}
  with probability greater than $1-\delta_u-M_4$, 
  where $\bar{S}(\delta_u)$ and $M_4$  are defined in (\ref{eq:def_M4}) and (\ref{eq:def_barS}), respectively, and
  $\mathfrak{G}$ and $c$ in (\ref{eq:S_FF}) and (\ref{eq:Sm_FF}) denote the sub-Gaussian norm of the standard normal distribution and
   a positive absolute constant, respectively.
  
\end{theorem}

\begin{figure*}[b]
    \centering
    \rule{\textwidth}{0.6pt}
    
    \begin{align}
    &A_K\Sigma +\Sigma A_K^\top +W+BK_P V(BK_P)^\top=O\label{eq:Pcont_Lyapunov}\\
       & M_1=\frac{\|Z\|_2 }{\lambda_{\min}(Z)} \max\left\{
        \frac{2\|C\|_2 (D_x + \|A_K^{-1}\|_2 \|B\|_2 \|u^\star\| )}{\sigma_p(CA_K^{-1}B)},\,
        \frac{
        8\sqrt{2mp} \| C\|_2^2 \|A_K^{-1}\|_2^2 \|B\|_2^2\|u^\star\|
        }{\sigma_p(CA_K^{-1}B)}
        \right\} \label{eq:def_M1} 
        \\
        &D_x= \|\mathbb{E}[e_x(0)]\| \notag
        \\
        &M_2 =\frac{\|Z\|_2^2 \|C\|_2^2 }{\lambda_{\min}(Z)^2}\max\left\{  
        \frac{ \mathrm{tr}(\Sigma_0)}{\mathrm{tr}(C\Sigma C^\top)}
        ,\frac{\|\Sigma_0\|_\mathrm{F} }{\|C\Sigma C^\top\|_\mathrm{F} }
        \right\} \label{eq:def_M2}\\
       &M_3=\frac{\|Z\|_2 \|C\|_2}{\lambda_{\min}(Z)} \max\left\{
        2\sqrt{2mp} \|C\|_2 \|A_K^{-1}\|_2^2 \|B\|_2^2  ,\,
        \frac{D_x +\|A_K^{-1}\|_2 \|B\|_2 \|u^\star \| }{ \|y^\star\| }
        \right\}\label{eq:def_M3} \\
       & \bar{S}(\delta_u)=\frac{1}{\sigma_p(CA_K^{-1}B)}\left(
       4\sqrt{2}\|C\|_2 \|A_K^{-1}\|_2 \|B\|_2 \|y^\star \|S_m(\delta_u/2)
        + 2\left(1+ \sqrt{2}\|C\|_2 \|A_K^{-1}\|_2 \|B\|_2S_m(\delta_u/2) \right) 
         S(\delta_u/2) 
         \right)
         \label{eq:def_barS}\\
        &S(\delta_u)=\sqrt{2\mathrm{tr}\left(C \Sigma C^\top \right)+
        \frac{9(\sqrt{2c}+2)\mathfrak{G}^2 \|C \Sigma C^\top \|_{\mathrm{F}} }{c^2}\log\frac{2}{\delta_u} 
        }\label{eq:S_FF}\\
        &S_m(\delta_u)=\sqrt{2m\mathrm{tr}\left(C \Sigma C^\top \right)+
        \frac{9(\sqrt{2mc}+2)\mathfrak{G}^2 \|C \Sigma C^\top \|_{\mathrm{F}} }{c^2}\log\frac{2}{\delta_u} 
        }
        \label{eq:Sm_FF}\\
        &M_4=2\exp\left(-c^2 
        \frac{
        \sigma_p(E^\star)^2/16-2m\mathrm{tr}(C\Sigma C^\top) 
        }{
        9(\sqrt{2mc}+2)\mathfrak{G}^2 \|C \Sigma C^\top \|_{\mathrm{F} }
        } \right)
        \label{eq:def_M4}
    \end{align}
    
\end{figure*}
\begin{proof}
See Appendix \ref{appendix:proof_of_SizeTauU}.
\end{proof}
\noindent Theorem~\ref{thm:size_of_tau_u} is used in the derivation of Theorem~\ref{thm:Samp_comp_gain}, which provides the theoretical analysis of the PI gain tuning method presented in Algorithm~\ref{algo:projected_grad_gain}.

\begin{remark}
    The condition in Theorem~\ref{thm:size_of_tau_u}, $\sigma_p(C A_K^{-1} B)\geq 4\sqrt{2m\mathrm{tr}(C\Sigma C^\top) }$, requires that the input/output signal strength is sufficiently larger than the magnitude of the noise.
    This requirement is analogous to the phenomenon that
    system identification accuracy deteriorates when the signal-to-noise ratio is small \cite{Oymak_Ozay}.
\end{remark}

\begin{remark}\label{rem:Uncertainty_FF}
    The second term in (\ref{eq:bound_hatu_uncertainty}) arises from stochastic variability of the noise, and it can be small when the matrix $C$ is sparse.
    From (\ref{eq:def_barS}), the dominant contribution to $\bar{S}(\delta_u)$ is given by
    \begin{equation}\label{eq:dominant_term_BarS}
        \frac{
        2 \sqrt{2}\|C\|_2 \|A_K^{-1}\|_2 \|B\|_2S_m(\delta_u/2)
         S(\delta_u/2) 
        }{\sigma_p(CA_K^{-1}B)}
    \end{equation}
    The product $S_m(\delta_u/2)
         S(\delta_u/2) $ satisfies
    \begin{equation*}
        \begin{split}
           &\quad S_m(\delta_u/2)S(\delta_u/2)\\
        &\leq  \sqrt{m}\left(
    2\mathrm{tr}\left(C \Sigma C^\top \right)
    +\frac{9(\sqrt{2c}+2)\mathfrak{G}^2 \|C \Sigma C^\top \|_{\mathrm{F}} }{c^2}\log\frac{2}{\delta_u}
        \right) .
        \end{split}
    \end{equation*}
    Since both $\mathrm{tr}(C \Sigma C^\top)$ and $\|C \Sigma C^\top \|_{\mathrm{F}}$ scale as $O(n^2 p)$,
    the overall dependency of (\ref{eq:dominant_term_BarS}) on the system dimensions and the uncertainty is $O\left(n^{4}mp^{3/2}(1 +\log(\frac{1}{\delta_u}) \right)$. 
    We note that this dependence is substantially reduced when the matrix $C$ is sparse, which is standard in many control areas, including large-scale systems \cite{Lunze_1992}.
   Let $s$ denote the number of nonzero components of $C$.
    Then the dependencies of $\mathrm{tr}(C \Sigma C^\top)$ and $\mathrm{tr}(C \Sigma C^\top)$ are $O(s)$, which implies that (\ref{eq:dominant_term_BarS}) reduces to $O\left(n^{3/2}m s^{3/2} (1 +\log(\frac{1}{\delta_u}) \right)$.
\end{remark} 
\begin{remark}\label{rem:Estimate_tau_u}
    The lower bound on the time horizon $\tau_u$ (\ref{eq:bound_of_tau_u}) involves quantities that depend on the unknown system parameters.
    However, an approximate lower bound ignoring log factors can be obtained from the input/output data in Step~1 of Algorithm~\ref{algo:DataMatrix}.
    By \cite[Lemma 12]{Convergence_Complexity}, we obtain $\| \exp{(A_K t)}\|_2= O\left(\exp\left(-\frac{t}{2\|Z\|_2}\right) \right)$, 
    which means that $\frac{1}{2\|Z\|_2}$ corresponds to the slowest mode of the closed-loop dynamics.
    In Step~1 of Algorithm~\ref{algo:DataMatrix}, $\mathbb{E}[y(t)-y(\infty)]=C\exp{(A_K t)} (\mathbb{E}[x_0]-A_K^{-1}BK_P^\prime y^\star)$ holds,
    and therefore, 
    for sufficiently large $t$, $\|\mathbb{E}[y(t)-y(\infty)]\|$ is approximated $\approx c^\prime \exp \left(-\frac{t}{2\|Z\|_2}\right)$, where $c^\prime$ is a constant depending on $A,\,B,\,C$, and $y^\star$.
    By simulating the system until sufficiently large time $\tau_{\mathrm{large}}$ in Step~1 of Algorithm~\ref{algo:DataMatrix} and measuring $\|y(t)-y(\tau_{\mathrm{large}})\|$, we can estimate $\|Z\|_2$.
    In the rest part of Algorithm~\ref{algo:DataMatrix}, we can simulating the system with control horizon $\tau_u$, which is determined by the approximate lower bound of (\ref{eq:bound_of_tau_u}).
\end{remark}

\subsection{PI Gain Tuning}\label{subsec:Method_GainTuning}
In this subsection, we consider Problem \ref{problem:gain_optimization}. 
We employ a model-free projected gradient descent with gradient estimation and provide a theoretical analysis on the required sample size, control time horizon, and the gradient estimation error.
Furthermore, our analysis also quantifies how the estimation error of the feedforward $\hat{u}_0$, computed in Algorithm~\ref{algo:DataMatrix}, affects the performance of the resulting PI controller.
\par
Under Assumption~\ref{assump:model-free}, the gradient of the cost function $f(K)$ in Problem~\ref{problem:gain_optimization} is not directly accessible.
Therefore, we apply a zeroth-order optimization method \cite{Takakura_Sato_24,Fazal_18,Convergence_Complexity_IEEE}, which uses the estimated gradient.
Based on these studies, we propose Algorithm \ref{algo:estimation_grad_cost} to estimate the gradient $\nabla f(K)$.
Note that $\hat{u}_0$, computed by Algorithm~\ref{algo:DataMatrix}, was employed as one of the inputs to Algorithm~\ref{algo:estimation_grad_cost}.
\begin{algorithm}
 \begin{algorithmic}[1]
 \caption{Gradient Estimation}\label{algo:estimation_grad_cost}
            \REQUIRE $K,\,\hat{u}_0,\,N>0,{\,N_\mathrm{sub}>0 },\,\tau>0,\,r>0$.
            \FOR{$i=1,\ldots,\,N$}
            \STATE Sample $U^i \in \mathbb{R}^{m \times {2p} } $
            from the uniform distribution $\mathcal{S}$ over matrices with $\|U^i\|_{\mathrm{F}}=\sqrt{2mp}$. 
            \STATE  Set the input (\ref{eq:2DoF}) using the {PI} gain
            { $K^{i,1}=K+r U^i ,\, K^{i,2}=K-r U^i$}, the feedforward $\hat{u}_0$, and $z(0)=0$.
           { \FOR{$j=1,\ldots,\,N_{\mathrm{sub}}$} 
            \STATE For $k\in {1,2}$, 
            simulate the system (\ref{eq:dynamics}) up to time $\tau$ with the gain $K^{i,k}$, and obtain $\hat{f}^{i,j,k}=  e(\tau)^\top Q_1 e(\tau)+ z(\tau)^\top Q_2 z(\tau) $.
            \ENDFOR }
           { \STATE For $k\in {1,2}$, compute $f^{i,k}=\frac{1}{N_{\mathrm{sub}}}\sum_{j=1}^{N_{\mathrm{sub}}} \hat{f}^{i,j,k}$ }
            \ENDFOR 
            \STATE Compute 
            
            \begin{equation}\label{eq:nabla_hat}
                \hat{\nabla} f(K;\hat{u}_0)=\frac{1}{2rN}\sum_{i=1}^N (f^{i,1}-f^{i,2}) U^i.
            \end{equation} 
            \RETURN $\hat{\nabla} f(K;\hat{u}_0)$.
        \end{algorithmic}
\end{algorithm}\\

\begin{remark}
   Fallah et al. \cite{Fallash_etal_2025} developed a model-free policy gradient method for systems with additive Gaussian noise.
    Although they consider the Linear-Quadratic-Gaussian (LQG) problem under assumptions similar to Assumption~\ref{assump:model-free},
    they assume access to an oracle that returns the true value of the cost function, which is unrealistic in practice.
    To the best of our knowledge, few studies have investigated gradient estimation method for LQG-type problems using only input/output data.
\end{remark}

The following theorem establishes that
the estimated gradient $\hat{\nabla} f(K;\hat{u}_0)$ approximates the true gradient $\nabla f(K)$
 up to a term depending on $\|\hat{u}_0-u^\star \|$,
with high probability.
\begin{theorem}\label{thm:Samp_comp_gain}
We define the sublevel set by $S(a)\coloneqq\{K\in \mathbb{R}^{m\times 2p}\,|\,f(K)\leq a +\mathrm{tr}(Q_1V)\}$.
For any $\epsilon^\prime>0$ and $\delta >0$, set  $r=O( {\epsilon^\prime}^{1/2} )$, $ \tau = O(\log\frac{1}{\epsilon^\prime})$, $ N=O(\frac{1}{{\epsilon^\prime}^{2}} \log \frac{1}{\delta})$, and $N_\mathrm{sub}=O(\frac{1}{ {\epsilon^\prime}^3} (\log \frac{1}{\delta})^2 )$. 
Let $\hat{u}_0$ be the feedforward input computed from Algorithm \ref{algo:DataMatrix} with $\tau_u=O(\log\frac{1}{ \epsilon^\prime})$,
and use it as the input to Algorithm \ref{algo:estimation_grad_cost}.
Then, with probability greater than $1-\delta-\delta_u-M_4$, we have
$\left\| \nabla f(K)-\hat{\nabla} f(K;\hat{u}_0) \right\|_{\rm{F}}
=O(\epsilon^\prime+1+ (\log\frac{1}{\delta_u})^2 )$, for any $K\in S(a)\cap \Omega$,
where  $M_4$ is defined in (\ref{eq:def_M4}). 
\end{theorem}
\begin{proof}
    See Appendix \ref{appendix:proof_of_SampCompGain}.
\end{proof}
In Theorem~\ref{thm:Samp_comp_gain}, the gradient estimation error $\left\| \nabla f(K)-\hat{\nabla} f(K;\hat{u}_0) \right\|_{\rm{F}}$ decomposes into two components:
an arbitrarily small term $O(\epsilon^\prime)$, and a residual term $O(1+ (\log \frac{1}{\delta_u} )^2)$, which cannot be made arbitrary small.
The latter term arises from the feedforward estimation error $\|\hat{u}_0-u^\star\|$, which is characterized in Theorem~\ref{thm:size_of_tau_u}.
As revealed in the proof of Theorem~\ref{thm:Samp_comp_gain}, the performance of the gradient estimation error scales $O(\|\hat{u}_0-u^\star\|^2)$.
In the model-free setting, we cannot obtain an accurate value of the equilibrium input $u^\star$ by solving (\ref{eq:equilibrium}).
The unavailability of $u^\star$ affects the sampling of the input/output data in the PI gain tuning, resulting in degradation of the PI gain tuning.
In contrast, when the model parameters $A,\,B,\,C,\,W,\,V,\, \mathcal{D}$ are accessible, we can obtain $u^\star$ accurately by solving (\ref{eq:equilibrium}), allowing the design of a 2DOF controller without considering the estimation error of $u^\star$ \cite{Youla_et_al_1985}.

\begin{remark}
    In Theorem~\ref{thm:Samp_comp_gain}, theoretical bounds on the parameters $\tau_u,\,\tau,\,r,\,N$, and $N_\mathrm{sub}$ depend on unknown system parameters.
   The bound on $\tau_u$ is derived from Theorem~\ref{thm:size_of_tau_u}, which provides an analysis on the control horizon $\tau_u$ and the estimation error of the feedforward.
    However, as discussed in Remark~\ref{rem:Estimate_tau_u}, the dominant factor in this bound, namely $\|Z\|_2$, can be estimated from a single simulation in Algorithm~\ref{algo:DataMatrix}, which makes the theoretical result practically usable.
    In contrast, the theoretical bounds on $\tau,\,r,\,N,$ and $N_\mathrm{sub}$ involve unknown constants, which limits the practical applicability.
    However, Theorem~\ref{thm:Samp_comp_gain} provides meaningful insights for the selection of these parameters.
    For example, the inner loop sample size $N_\mathrm{sub}$ should exceed the outer loop sample size $N$ in Algorithm~\ref{algo:estimation_grad_cost}, 
    and that the control horizon $\tau$ grows only logarithmically with respect to $\epsilon^\prime$.
    Note that in the literature on policy gradient methods for feedback optimization \cite{Fazal_18,Convergence_Complexity_IEEE,Takakura_Sato_24}, it is standard that the sample complexity result depends on unknown system parameters. 
\end{remark}

To address Problem~\ref{problem:gain_optimization},
we adopt Algorithm~\ref{algo:projected_grad_gain}, which is based on \cite[Algorithm~2]{Takakura_Sato_24}.
We assume that the initial point $K^0$ of Algorithm~\ref{algo:projected_grad_gain} lies in $\mathcal{S}\cap \Omega$.
We note that although the set $\mathcal{S}\cap \Omega$ can be nonconvex, \cite[Theorem~3]{Takakura_Sato_24} ensures that
$K^i$ remains in $\mathcal{S} \cap \Omega$ after projection onto the convex set $\Omega$.
According to \cite[Theorem~3]{Takakura_Sato_24},
for any $ K^0 \in \mathcal{S}\cap \Omega $ and for $\epsilon=O(1+\log\frac{1}{\delta_u})$,
the sequence 
$\{K^i\}_{i=1}^{T^\prime}$ converges to an $\epsilon$-stationary point with probability greater than $1-T^\prime(\delta+\delta_u + M_4)$, where $T^\prime$ denotes the number of iterations in Algorithm \ref{algo:projected_grad_gain}.
Since gradients cannot be estimated to arbitrary accuracy due to the feedforward estimation error, we can only obtain an $\epsilon$-stationary with $\epsilon=O(1+\log\frac{1}{\delta_u})$.
The iteration complexity of the algorithm is $T^\prime=O(1/\epsilon^2)$, and the total sample complexity is $2NN_\mathrm{sub}T^\prime\tau=O\left(\frac{1}{\epsilon^7} \log\frac{1}{\epsilon}\right)$.
Finally, we recall the definition of an $\epsilon$-stationary point below:

\begin{definition}\label{def:stationary_point}{\cite[Definition 1]{Takakura_Sato_24}}
    For a positive constant $\eta$ and $\epsilon$, $K$ is called an $\epsilon$-stationary point if $\|G_\eta (K)\|_{\mathrm{F}}\leq \epsilon$, where $K^+ \coloneqq \mathrm{proj}_\Omega (K-\eta\nabla f(K))$ and $G_\eta(K)\coloneqq\frac{1}{\eta}(K^+-K)$.
\end{definition}

\begin{algorithm}
 \begin{algorithmic}[1]
 \caption{Projected gradient method \cite[Algorithm 2]{Takakura_Sato_24} }\label{algo:projected_grad_gain}
            \REQUIRE $K^0 \in \mathcal{S}\cap \Omega$, $\hat{u}_0$, $N>0$,\,$N_\mathrm{sub}>0$, $\tau>0$, $r>0$, $T>0$, $\eta >0$, $ \epsilon>0$.
            \FOR{$i=1,\ldots$, $T$}
            \STATE Compute $\hat{\nabla} f(K^i;\hat{u}_0)$ using Algorithm \ref{algo:estimation_grad_cost}.
            \STATE $K^{i+1}=\mathrm{proj}_\Omega \left(K^i-\eta\hat{\nabla}f(K^i; \hat{u}_0)\right)$.
            \IF{$\|K^{i+1} -K^{i} \|_{\rm{F}}\leq \epsilon \eta $}
            \RETURN $K^i$.
            \ENDIF
            \ENDFOR
            \RETURN $K^T$.
        \end{algorithmic}
\end{algorithm}

\begin{remark}
    The assumption that $ K^0 \in \mathcal{S}\cap \Omega$ in Algorithm~\ref{algo:projected_grad_gain} is restrictive.
    However, such an assumption is standard in the literature on policy gradient methods for feedback gain optimization \cite{Convergence_Complexity_IEEE, Takakura_Sato_24}, and we adopt the same assumption here.
    Moreover, for certain classes of systems, stabilizing gains can be constructed in the model-free setting, as noted in Remark~\ref{rem:Pgain_stabilizing}.
    All PI gains in Algorithm~\ref{algo:estimation_grad_cost} and \ref{algo:projected_grad_gain} remain in $\mathcal{S}$ with high probability.
    Lemma~\ref{lem:bound_of_smoothing} in Appendix~\ref{appendix:proof_of_SampCompGain} establish that $K\pm rU^i$ in Algorithm~\ref{algo:estimation_grad_cost} lie in $\mathcal{S}$ if the smoothing parameter $r$ is sufficiently small.
    Furthermore, since $f(K^i)$ is monotonically decreasing \cite[Theorem 3]{Takakura_Sato_24}, $K^i$ stabilizes the closed-loop system under Assumption~\ref{assump:detectability}.
\end{remark}

\section{Comparison with model-based method}\label{sec:Compare_IndirectApproach}

In this section, we provide a theoretical comparison between the proposed model-free design of 2DOF PI controllers and a model-based method.
The model-based method refers to an approach in which we first identify the system parameters and subsequently design a controller based on the identified model.
We compare these two approaches in terms of sample and computational complexity.
The result is summarized in Table~\ref{tab:comparison_ModelBased}.
Throughout this section, we assume that the matrix $A$ in (\ref{eq:dynamics}) is Hurwitz for simplicity.
\par
\begin{table*}[b]
\rule{\textwidth}{0.6pt}
\centering
\caption{Comparison between the proposed method and the model-based method. }
\label{tab:comparison_ModelBased}
\setlength{\tabcolsep}{6pt}
\renewcommand{\arraystretch}{1.2}

\begin{tabular}{@{}l c cc@{}}
\toprule
 & \textbf{Feedforward} & \multicolumn{2}{c}{\textbf{PI gain}} \\
\cmidrule(lr){2-2} \cmidrule(lr){3-4}
 &\textbf{Sample size}   & \textbf{Sample size} & \textbf{Computational cost} \\
\midrule
\textbf{Proposed method}
 & $O(\log \frac{1}{\epsilon})$ for $\epsilon=O(1+\log^2\frac{1}{\delta_u})$ 
 & $\tilde{O}(\frac{1}{\epsilon^7})$ for $\epsilon=O(1+\log^2\frac{1}{\delta_u})$
 & cost of $\mathrm{proj}_\Omega$, or $O(p^2)$ \\
\textbf{Model-based method}
 & $\tilde{O}(\frac{1}{\epsilon^2})$
 & $\tilde{O}(\frac{1}{\epsilon^2})$
 & $O((n+p)^3)$\\
\bottomrule
\end{tabular} 
\end{table*}

In a model-based method, one first identifies the discrete-time model using input/output data collected from the continuous-time system (\ref{eq:dynamics}).
A controller is then designed for this identified discrete-time model, and implemented on the original continuous-time system (\ref{eq:dynamics}), via the zero-order hold method \cite{Hubert_Raphael_1972}.
The zero-order hold implementation yields the following 2DOF PI controller:
\begin{equation}\label{eq:2DOF_zoh}
\begin{split}
    u(t) &=K_{P}e_k + K_I z_k +u_0,\, (hk\leq t < h(k+1))\\
    z_{k+1}&=z_k +\,e_k,
\end{split}
\end{equation}
where $e_k\coloneqq e(hk)$ and $h>0$ denotes the constant sampling interval. 
When applied to the original system (\ref{eq:dynamics}), the closed-loop dynamics of $e_{x,k}\coloneqq e_x(hk)$ is given by
\begin{equation}\label{eq:dynamics_zoh_closedloop}
    \begin{split}
        \begin{pmatrix}e_{x,k+1}\\ z_{k+1}\end{pmatrix}& =
      \bar{A}_{K,\mathrm{d}}
        \begin{pmatrix}e_{x,k}\\ z_{k}\end{pmatrix} 
        +\bar{B}_\mathrm{d} (u_0-u^\star)
        +\begin{pmatrix}w_k + B_\mathrm{d} K_P v_k  \\ v_k\end{pmatrix} \\
        e_k&=-C e_{x,k}+v_k
    \end{split}
\end{equation}
where $v_k \coloneqq v(hk)$,
$w_k$ is zero-mean Gaussian noise with covariance matrix $W_\mathrm{d} = \int_0^h \exp(At)W\exp(A^\top t)dt$, and we define
\begin{equation*}
\begin{split}
&\bar{A}_{K,\mathrm{d}}=\bar{A}_\mathrm{d}-\bar{B}_\mathrm{d} K\bar{C},\,\bar{A}_\mathrm{d}=\begin{pmatrix}A_\mathrm{d} &O \\ -C A_\mathrm{d} & I\end{pmatrix} ,\,
\bar{B}_\mathrm{d}=\begin{pmatrix}B_\mathrm{d} \\O\end{pmatrix}\\
    &A_\mathrm{d}=\exp(Ah),\, B_\mathrm{d}=(\exp(Ah)-I)A^{-1}B.
    \end{split}
\end{equation*} 
Throughout this section, let $A_\mathrm{Id},\,B_\mathrm{Id},\,C_\mathrm{Id},\,W_\mathrm{Id},\,V_\mathrm{Id}$ denote
the identified counterparts of $A_\mathrm{d},\,B_\mathrm{d},\,C,\,W_\mathrm{d},\,V$, respectively.

\subsection{Feedforward Design }

To facilitate the comparison of sample complexity for the feedforward design, we first describe how the feedforward is computed in the model-based method.
In the steady-state of the zero-order hold implemented discrete-time model, it holds that
$x^\star=A_{d} x^\star + B_{d}u^\star, \, y^\star = C x^\star$, which is equivalent to (\ref{eq:equilibrium}).
Assumption~\ref{assump:dynamics} implies that the matrix $\begin{pmatrix}
        A_\mathrm{d}-I& B_\mathrm{d} \\ C &O
    \end{pmatrix}$ is of full row rank,
and hence the equilibrium feedforward $u^\star_\mathrm{d}$ is given by
\begin{equation}\label{eq:equib_discrete}
    \begin{pmatrix}x^\star_\mathrm{d} \\u^\star_\mathrm{d}\end{pmatrix} =
    \begin{pmatrix}
        A_\mathrm{d}-I& B_\mathrm{d} \\ C &O
    \end{pmatrix}^+
     \begin{pmatrix}0\\y^\star\end{pmatrix}.
\end{equation}
Similarly, the estimated feedforward $u^\star_\mathrm{Id}$ is obtained by
\begin{equation}\label{eq:equib_identified}
    \begin{pmatrix}x^\star_\mathrm{Id} \\u^\star_\mathrm{Id}\end{pmatrix} =
    \begin{pmatrix}
        A_\mathrm{Id}-I& B_\mathrm{Id} \\ C_\mathrm{Id} &O
    \end{pmatrix}^+
     \begin{pmatrix}0\\y^\star\end{pmatrix}.
\end{equation}
To evaluate the performance of the feedforward $u_0$, we consider the steady-state output error $\|y^\star -\mathbb{E}[y_{u_0}(\infty)]\|$, where $y_{u_0}(t)$ is the output trajectory under $u(t)=u_0$.
Since $\mathbb{E}[y_{u_0}(\infty)] =-CA^{-1}_\mathrm{d} B_\mathrm{d} u_0 $ and $y^\star=-CA^{-1}_\mathrm{d} B_\mathrm{d} u^\star_\mathrm{d}$ hold, it follows that
\begin{equation}\label{eq:performance_u0}
    \begin{split}
         \|y^\star - \mathbb{E}[y_{u^\star_{\mathrm{Id} }}(\infty)] \| &=
        \|CA_\mathrm{d}^{-1} B_\mathrm{d} u^\star_\mathrm{d}  - CA_\mathrm{d}^{-1} B_\mathrm{d} u^\star_{\mathrm{Id} } \| \\
        &\leq \|C A_\mathrm{d}^{-1} B_\mathrm{d} \|_2 \| u^\star_\mathrm{d} - u^\star_{\mathrm{Id} } \|.
    \end{split}
\end{equation}
We now bound the estimation error $ \| u^\star_\mathrm{d} - u^\star_{\mathrm{Id} } \|$.
Assume that the estimated errors of system matrices satisfy $\max \{ \|A_\mathrm{d}- A_{\mathrm{Id}}\|_2,\,\|B_\mathrm{d}-B_{\mathrm{Id}}\|_2,\, \|C-C_{\mathrm{Id}}\|_2 \} \leq \epsilon_{\mathrm{Id}}$,
where the freedom of similarity transformation is omitted for simplicity.
Then, \cite[Theorem 5.1]{Wedin_1973} implies $\| u^\star_\mathrm{d} - u^\star_{\mathrm{Id} } \| =O(\epsilon_{\mathrm{Id}})$.
Consequently, $ \|y^\star - \mathbb{E}[y_{u^\star_{\mathrm{Id} }}(\infty)] \| =O(\epsilon_\mathrm{Id})$ holds.
When the system matrices are identified via the Ho-Kalman method,
Oymak and Ozay \cite{Oymak_Ozay} showed that
the identified errors satisfy $\epsilon_\mathrm{Id} =O\left( \frac{1}{\sqrt{N_\mathrm{Id}} } \right)$, ignoring log factors,
where $N_\mathrm{Id}$ is the length of sampled input/output data.
Consequently, to achieve $\|y^\star - \mathbb{E}[y_{u^\star_{\mathrm{Id} }}(\infty)] \|=O(\epsilon)$,
the model-based method requires the $O \left(1/ \epsilon^2  \right)$ samples.\par
The sample complexity advantage of the proposed method over the model-based method depends on the performance level required for the steady-state output.
Using the same argument as in (\ref{eq:performance_u0}), we have $\|y^\star - \mathbb{E}[y_{\hat{u}_0 } (\infty)] \| \leq \|CA^{-1}B\|_2 \|\hat{u}_0 -u^\star\|$.
According to Theorem~\ref{thm:size_of_tau_u}, for any $\epsilon_u$, if we set the control horizon $\tau_u=O(\log\frac{1}{\epsilon_u})$,
Algorithm~\ref{algo:DataMatrix} yields the feedforward $\hat{u}_0$ satisfying $\|\hat{u}_0 -u^\star\| \leq \epsilon_u+\bar{S}(\delta_u)$ with probability greater than $1-\delta_u-M_4$, where $\bar{S}(\delta_u)$ and $M_4$ are defined in (\ref{eq:def_barS}) and (\ref{eq:def_M4}), respectively.
Ignoring log factors $\log\frac{1}{\delta_u}$, we have 
$\bar{S}(\delta_u)\simeq 
\frac{2 \sqrt{2m}\|C\|_2 \|A_K^{-1}\|_2 \|B\|_2
\left(
    2\mathrm{tr}\left(C \Sigma C^\top \right)
    +9(\sqrt{2c}+2)\mathfrak{G}^2 \|C \Sigma C^\top \|_{\mathrm{F}} /c^2
        \right) 
        }{\sigma_p(CA_K^{-1}B)}
        \eqqcolon M_5$,
where $M_5$ is the dominant term discussed in Remark~\ref{rem:Uncertainty_FF}.
If a performance level of $\epsilon=\|y^\star - \mathbb{E}[y_{\hat{u}_0 } (\infty)] \|
\simeq \|CA^{-1}B\|_2 M_5$ is acceptable,
Algorithm~\ref{algo:DataMatrix} obtains such a feedforward with the control horizon $O\left(\log \left(1/\epsilon\right)\right)$,
which is of smaller order than $O \left( 1/ \epsilon^2  \right)$ samples required by the model-based method.
On the other hand, if we pursue a higher-accuracy feedforward
$\epsilon=\|y^\star - \mathbb{E}[y_{u_0} (\infty)] \|
\ll \|CA^{-1}B\|_2 M_5 $,
we need to employ the model-based method with $O \left(1/ \epsilon^2  \right)$ samples.
Finally, we remark that such extremely high performance feedforward is not essential in a 2DOF PI controller because 
integral controller guarantees $\mathbb{E}[y(t)]\to y^\star$.

\subsection{PI gain tuning }
For the purpose of the comparison with the model-based method, we formulate the PI gain tuning problem for the identified model, and analyze it in terms of sample and computational complexity.
As in the previous subsection, we first introduce the PI gain tuning problem for a discrete-time system.
The problem is formulated as
\begin{equation}\label{prob:PI_discrete}
    \begin{split}
        \min_{K \in \Omega} \quad & f_\mathrm{d} (K) \coloneq
\lim_{T_d\to \infty}\frac{1}{T_d}\mathbb{E}\left[ \sum_{k=1}^{T_d} e_k^\top Q_1 e_k+ z_k^\top Q_2 z_k 
\right] \\
        \mathrm{s.t.} \quad &(\ref{eq:dynamics_zoh_closedloop}),
        u_0=u^\star_\mathrm{d},\,
        \,\text{$\bar{A}_{K,\mathrm{d}}$ is Schur stable,}
    \end{split}
\end{equation}
where $\bar{A}_{K,\mathrm{d}}$ is the closed-loop matrix $\bar{A}_{K,\mathrm{d}}=\bar{A}_\mathrm{d}-\bar{B}_\mathrm{d} K\bar{C}$ and $u^\star_\mathrm{d}$ is defined in (\ref{eq:equib_discrete}).
The cost function $f_\mathrm{d} (K)$ can be written as
$f_\mathrm{d}(K)=\mathrm{tr}(X_\mathrm{d} Q^\prime)+\mathrm{tr}(V Q_1)$, where $X_\mathrm{d}$ is the solution to the Lyapunov equation $\bar{A}_{K,\mathrm{d} } X_\mathrm{d} \bar{A}_{K,\mathrm{d} }^\top -X_\mathrm{d} + \begin{pmatrix} W_d + B_\mathrm{d} K_P V (B_\mathrm{d} K_P)^\top & B_\mathrm{d} K_P V \\ (B_\mathrm{d} K_P V)^\top& V\end{pmatrix}=O$.
Finally, we define the corresponding PI gain tuning problem for identified model:
\begin{equation}\label{prob:PI_Identified}
    \begin{split}
        \min_{K \in \Omega} \quad & f_{\mathrm{Id}}(K) =\mathrm{tr}(X_{\mathrm{Id}} Q^\prime_{\mathrm{Id}}) +\mathrm{tr}(Q_1 V_{\mathrm{Id}}) \\
        \mathrm{s.t.} \quad &\ \bar{A}_{K,\mathrm{Id}} X_{\mathrm{Id}} \bar{A}_{K,\mathrm{Id}}^\top - X_{\mathrm{Id}} \\
        &+\begin{pmatrix} W_\mathrm{Id} + B_{\mathrm{Id}} K_P V_{\mathrm{Id}} (B_{\mathrm{Id}} K_P)^\top & B_{\mathrm{Id}} K_P V_{\mathrm{Id}} \\ (B_{\mathrm{Id}} K_P V_{\mathrm{Id}})^\top& V_{\mathrm{Id}}\end{pmatrix} =O,\\
        & \text{$\bar{A}_{K,\mathrm{Id}} \coloneqq \bar{A}_\mathrm{Id}-\bar{B}_\mathrm{Id} K \bar{C}_\mathrm{Id}$ is Schur stable.}
    \end{split}
\end{equation} \par

We next compare the proposed method with the model-based method in terms of sample complexity.
In the model-based method, 
we suppose that the identified system matrices satisfy
\begin{equation}\label{eq:assumption_identification}
    \begin{split}
        \max \{& \|A_\mathrm{d}- A_{\mathrm{Id}}\|_2,\,\|B_\mathrm{d}-B_{\mathrm{Id}}\|_2,\, \|C-C_{\mathrm{Id}}\|_2 ,\\
         &\|W_\mathrm{d}-W_{\mathrm{Id}} \|_2 
         ,\,\|V-V_{\mathrm{Id}}\|_2 \} 
        \leq \epsilon_\mathrm{Id},
    \end{split}
\end{equation}
where the freedom of similarity transformation is omitted for simplicity.
Then Theorem~\ref{thm:degradation_indirect} in Appendix~\ref{appendix:indirect_approach} implies that any
$\epsilon_\mathrm{Id}$-stationary point of (\ref{prob:PI_Identified}) is an $O(\epsilon_\mathrm{Id})$-stationary point of (\ref{prob:PI_discrete}), where an $\epsilon_\mathrm{Id}$-stationary point is defined in Definition~\ref{def:stationary_point}.
A rigorous finite-sample analysis of the model-based method is challenging, because, to the best of our knowledge, 
little is known about nonasymptotic identification errors for the covariance matrices $W_\mathrm{d}$, $V_\mathrm{d}$.
For the sake of comparison, however, suppose that some system identification method---for example, the Numerical Algorithms for Subspace State Space identification (N4SID) \cite{Overschee_2012_etal}---provides identified matrices satisfying (\ref{eq:assumption_identification}) with $\epsilon_\mathrm{Id}=O\left(\frac{1}{\sqrt{N_\mathrm{Id}} }\right)$ ignoring log factors, where $N_\mathrm{Id}$ is the length of the input/output data.
This rate matches the known sample complexity result of Ho-Kalman method for identifying only $A_\mathrm{d},\,B_\mathrm{d},\,C_\mathrm{d}$ \cite{Oymak_Ozay}.
Under this assumption, the model-based method requires $N_\mathrm{Id}=O(1/\epsilon_\mathrm{Id}^2)$ to obtain an $O(\epsilon_\mathrm{Id})$-stationary point of (\ref{prob:PI_discrete}).
The result in Section~\ref{subsec:Method_GainTuning} implies that for any $\epsilon=O(1+\log\frac{1}{\delta_u})$,
Algorithm~\ref{algo:projected_grad_gain} attains $\epsilon$-stationary point with the sample complexity $O\left(\frac{1}{\epsilon^7}\right)$, ignoring the log factor.
Thus, the model-based method achieves a more favorable worst-case sample complexity.
\par

In terms of computational cost, the proposed method is more efficient than the model-based method.
In the model-based method, applying the projected gradient descent to (\ref{prob:PI_Identified}) requires computing the gradient $\nabla f_{\mathrm{Id}}(K)$ at each iteration.
Following the derivation in \cite[Proposition 1]{Martensson_2012}, $\nabla f_{\mathrm{Id}}(K)$ is given by
\begin{equation}\label{eq:gradient_discrete_id}
    \begin{split}
        &\quad\nabla f_{\mathrm{Id}}(K) \\
        &= -2\bar{B}_\mathrm{Id}^\top Y_\mathrm{Id} \bar{A}_{K,\mathrm{Id}} X_\mathrm{Id} \bar{C}_\mathrm{Id}^\top
        +  2\bar{B}_\mathrm{Id}^\top Y_\mathrm{Id} \begin{pmatrix}
            B_\mathrm{Id} K V_\mathrm{Id} &O \\ V_\mathrm{Id} &O
        \end{pmatrix}, 
    \end{split}
\end{equation}
where $Y_\mathrm{Id}$ is the solution to $\bar{A}_{K,\mathrm{Id} }^\top Y_\mathrm{Id}  \bar{A}_{K,\mathrm{Id}} -Y_\mathrm{Id} +Q^\prime_\mathrm{Id}=O$.
Computing $\nabla f_{\mathrm{Id} }(K)$ requires not only several matrix multiplications but also solving two Lyapunov equations, resulting in a computational complexity of $O((n+p)^3)$ \cite{Barraud_1977},
which is prohibitive when the identified state dimension is large.
In contrast, our proposed method, Algorithm~\ref{algo:estimation_grad_cost} and~\ref{algo:projected_grad_gain},
has the dominant computational cost when computing $\hat{f}^{i,j,k}$ $\left(O(p^2) \right)$ and performing projection.
These costs are substantially smaller than the model-based method when the input and output dimensions $m$ and $p$ are small, which is typical in practical control applications.

\section{numerical experiments}\label{sec:numerical_experiment}
In this section, we present numerical experiments to demonstrate the effectiveness of the proposed method 
 and to compare its performance with that of the model-based method.
Throughout the numerical experiments, we considered the systems of dimension $n=20$ and $m=p=2$. 
The system matrices were constructed as $A=J-R$, $B=3\times\mathrm{randn}(n,m)$, $C=B^\top$, $W=10^{-2}I$ and $V=5\times 10^{-4} I$ where
$J=(\bar{J}-\bar{J}^\top)/2$, $ \bar{J}=\mathrm{randn}(n,n)$, $ R=\bar{R}\bar{R}^\top$, $\bar{R}=2\mathrm{randn}(n,n)$, and $\mathrm{randn}(a,b)$ denotes an $a\times b$ matrix with i.i.d. standard normal entries.
The distribution $\mathcal{D}$ was set to the uniform distribution over $[3,\,3]^{n}$.
The output set-point was fixed to $y^\star=\begin{pmatrix}
    5 & 5 
\end{pmatrix}^\top$.
In the model-based method, system identification is performed using the n4sid function in ControlSystemIdentification.jl package, and we set the sampling period $h=10^{-2}$.

\subsection{Feedforward Design}\label{subsec:experiment_FF} 
We numerically evaluated the performance of Algorithm \ref{algo:DataMatrix} and compared it with the model-based method.
We set the control horizon $\tau_u=108.93$, which corresponds to the value given by the right hand side of (\ref{eq:bound_of_tau_u}) with $\epsilon_u=10^{-3}$, and we chose $K_P^\prime=10^{-3}I$.
In the model-based method, we identified the system matrices from a single trajectory with $32{,}680$ time steps, which corresponds to $\frac{(m+1)\tau_u}{h}$,
and obtained the identified matrices $A_\mathrm{Id},\,B_\mathrm{Id},\, C_\mathrm{Id}$.
Then the feedforward $u^\star_\mathrm{Id}$ was subsequently computed by (\ref{eq:equib_identified}).
Both Algorithm~\ref{algo:DataMatrix} and the model-based method were run 10 times,
and their performance was assessed in terms of the steady-state output relative error $\frac{\|y^\star -\mathbb{E}[y_{u_0}(\infty)]\|}{\|y^\star \|} $, 
as illustrated in Fig.~\ref{fig:Boxplot_SteadyError_FF}.
The proposed method consistently achieved smaller steady-state output relative errors.
\begin{figure}[!t]
    \centering
    \includegraphics[width=0.65\columnwidth]{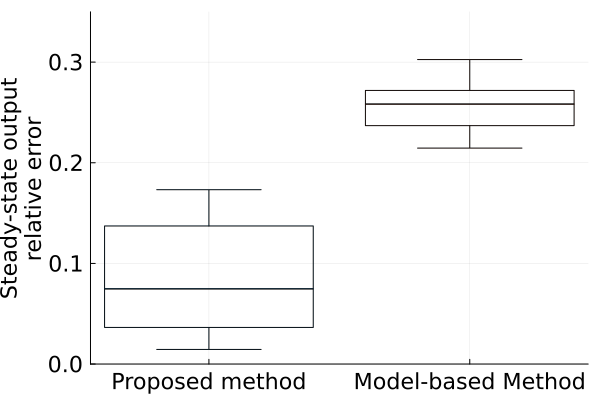} 
    \caption{Boxplots of the steady-state output relative errors $\frac{\|y^\star -\mathbb{E}[y_{u_0}(\infty)]\|}{\|y^\star \|} $ over 10 trials. }
    \label{fig:Boxplot_SteadyError_FF}
\end{figure}

\subsection{PI Gain Tuning}\label{subsec:experiment_PIgain}
We evaluated the performance of Algorithm~\ref{algo:projected_grad_gain} and numerically compared it with the model-based method. 
The numerical experiment in this subsection was carried out on the same system in Section~\ref{subsec:experiment_FF}.
The constraint set was chosen as $\Omega=\{ K=\begin{pmatrix} K_P &K_I\end{pmatrix}\,|\, \|K_P\|_\mathrm{F}\leq 5,\, \|K_I\|_\mathrm{F}\leq 5\}$.
The inputs to Algorithm \ref{algo:projected_grad_gain} were set to $K^0=\textcolor{black}{\begin{pmatrix}I & I \end{pmatrix} }$,
$N=15$, $N_\mathrm{sub}=20$, $ \tau=10$, $ r=0.09$, $ T=20$, $ \eta= 10^{-3}$, and $ \hat{u}_0$,
where $\hat{u}_0$ was computed in the setting of the preceding subsection.
We skipped the stopping test in Step~4 of Algorithm~3 to fix the number of iterations, 
and this allowed us to provide consistent comparisons with the model-based method under the fixed sample size.
The cost function weights were chosen as $Q_1=200I$ and $Q_2=20I$.
In the model-based method, the system matrices were identified from a single trajectory with $12{,}032{,}680$ time steps, which corresponded to $\frac{2NN_\mathrm{sub}T\tau+(m+1)\tau_u}{h}$, and then Problem~(\ref{prob:PI_Identified}) was solved.
The cost function weights were scaled to $Q_1=10^{-1}I$ and $Q_2=10^{-2}I$.
Problem~(\ref{prob:PI_Identified}) was solved using the projected gradient descent with the step size $10^{-5}$, running for $10^5$ iterations.
The initial point was chosen as $K^0=\begin{pmatrix}10^{-2}I & 10^{-2}I \end{pmatrix}$.
While updating the PI gain in the projected gradient descent, the PI gain remained in the set of stabilizing gains by ensuring the PI gain update remained sufficiently small, which was achieved by carefully scaling the cost weights $Q_1$, $Q_2$ and choosing the sufficiently small step size.
The feedforward $u^\star_\mathrm{Id} $ was computed by (\ref{eq:equib_identified}).

\begin{remark}
    Several inequalities in the proof of Theorem~\ref{thm:Samp_comp_gain}---(\ref{eq:size_OuterLoop}), (\ref{eq:bound_smoothing}), and (\ref{eq:prob_ineq_innerloop})---provide
    theoretical bounds on the parameters $r$, $N$ and $N_\mathrm{sub}$, thereby offering a principled way to select them.
    However, the resulting numerical values are exceedingly conservative due to the loose inequalities in the analysis.
    Consequently, these values are impractical for implementation, 
    and we employed smaller values of $N,\,N_\mathrm{sub}$ and a larger value of $r$ in the numerical experiments.
    Although the theoretical result does not yield usable nonasymptotic constants,
    it nevertheless provides meaningful insights in the asymptotic regime.
    For example, the inner loop sample size $N_\mathrm{sub}$ should be larger than the outer loop sample size $N$ in Algorithm~\ref{algo:estimation_grad_cost}.
\end{remark}

The proposed method achieved better control performance than the model-based method.
First, we compare the PI gain performance. 
To this end, we introduce a performance metric
$\bar{f}(K)=\frac{1}{N_\mathrm{eval}} \sum_{i=1}^{N_\mathrm{eval} }\left( \frac{1}{\tau_\mathrm{eval}} \int_0^{\tau_\mathrm{eval}} e(t)^\top Q_1 e(t)+ z(t)^\top Q_2 z(t) dt \right)$, 
where $e(t)$ and $z(t)$ were trajectories under the input (\ref{eq:2DoF}) in the proposed method or (\ref{eq:2DOF_zoh}) in the model-based method.
To isolate the effect of the uncertainty of the feedforward, the feedforward $u_0$ was chosen as $u^\star$.
Each method was run for 10 trials, and $\bar{f}(K)$ was computed in each trial with $N_\mathrm{eval}=200,\,\tau_\mathrm{eval}=300 $.
The result was shown in Fig.~\ref{fig:Boxplot_PI_Gain}.
The proposed method consistently achieved smaller values of $\bar{f}(K)$.
Secondly, we compared the 2DOF PI controller performance by examining the simulated output trajectories $y(t)$.
We simulated the system (\ref{eq:dynamics}) under the input (\ref{eq:2DoF}) in the proposed method or (\ref{eq:2DOF_zoh}) in the model-based method, 
using the corresponding feedforward inputs $\hat{u}_0$ and $u^\star_\mathrm{Id} $, respectively.
The results are demonstrated in Fig.~\ref{fig:Trajectory_2DOF}.
The proposed method exhibited smaller overshoot, that is, achieved better transient performance.

\begin{figure}[!t]
    \centering
    \includegraphics[width=0.65\columnwidth]{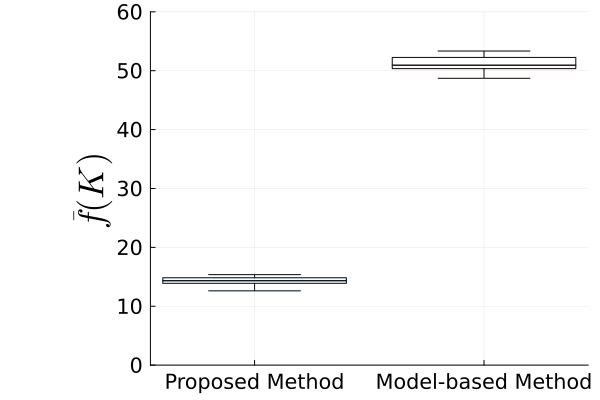} 
    \caption{Boxplots of $\bar{f}(K)$ over 10 trials with $N_\mathrm{eval}=200,\,\tau_\mathrm{eval}=300 $.
     }
    \label{fig:Boxplot_PI_Gain}
\end{figure}

\begin{figure}[!t]
    \centering
    \subfloat[]{\includegraphics[width=0.5\columnwidth]{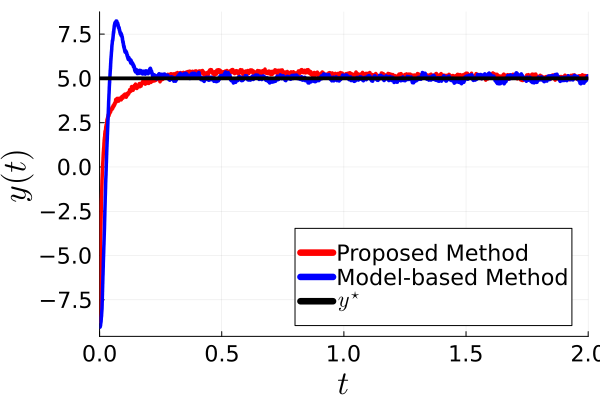}}\hfill
    \subfloat[]{\includegraphics[width=0.5\columnwidth]{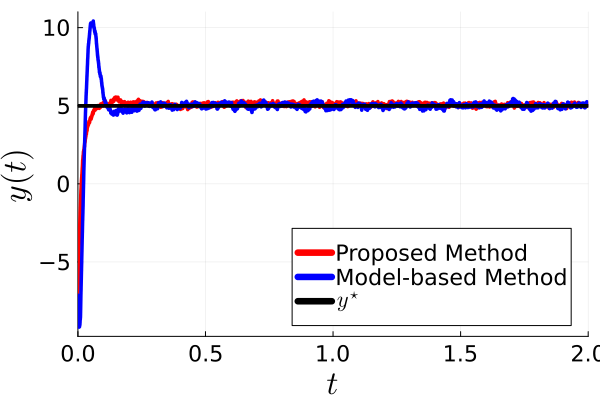}} 
    \caption{
    Trajectories of the output $y(t)$. The red line represents the trajectory under the controller designed by the proposed method.
    The blue line shows the trajectory under the controller designed by the model-based method.
    The black line corresponds to the desired output setpoint $y^\star$.
    (a) The first component of $y(t)$. (b) The second component of $y(t)$. }
    \label{fig:Trajectory_2DOF}
\end{figure}

\subsection{ Results across multiple systems }

In this subsection, we demonstrate that the proposed method consistently outperforms the model-based method across different system parameters.
We independently generated 10 systems as described at the outset of this section,
and conducted 10 trials for each system.\par
In the feedforward comparison, $\hat{u}_0$ and $u^\star_\mathrm{Id}$ were computed for each trial and each system.
All inputs to Algorithm~\ref{algo:DataMatrix} were kept fixed as in Section~\ref{subsec:experiment_FF}, 
except for $\tau_u$.
The value of $\tau_u$ was chosen as the right hand side of (\ref{eq:bound_of_tau_u}), which varied across systems, with $\epsilon_u=10^{-3}$.
Accordingly, the number of time steps for the system identification was set to $\frac{(m+1)\tau_u}{h}$.
For each system, we evaluated the average of steady-state output relative errors $\frac{\|y^\star -\mathbb{E}[y_{u_0}(\infty)]\|}{\|y^\star \|} $ over 10 trials, as illustrated in Fig.~\ref{fig:Boxplot_Average_FF}.
The proposed method consistently achieved smaller values of average steady-state output relative error.\par

In the PI gain comparison, the PI gains were computed for each trial and each system following the procedure in Section~\ref{subsec:experiment_PIgain}.
We compared both methods by evaluating $\frac{\bar{f}_\mathrm{MF}(K) }{ \bar{f}_\mathrm{MB}(K)}$, where $\bar{f}_\mathrm{MF}(K)$ and $\bar{f}_\mathrm{MB}(K)$ denote the average values of $\bar{f}(K)$ over 10 trials under the proposed and model-based method, respectively.
For each system, we computed the value of $\frac{\bar{f}_\mathrm{MF}(K) }{ \bar{f}_\mathrm{MB}(K)}$ with $N_\mathrm{eval}=200,\,\tau_\mathrm{eval}=300 $, as shown in Fig.~\ref{fig:Boxplot_Average_FF}.
In all 10 systems, $\frac{\bar{f}_\mathrm{MF}(K) }{ \bar{f}_\mathrm{MB}(K)}$ was below $0.4$, indicating that the proposed method consistently outperformed the model-based method.

\begin{figure}[!t]
    \centering
    \includegraphics[width=0.65\columnwidth]{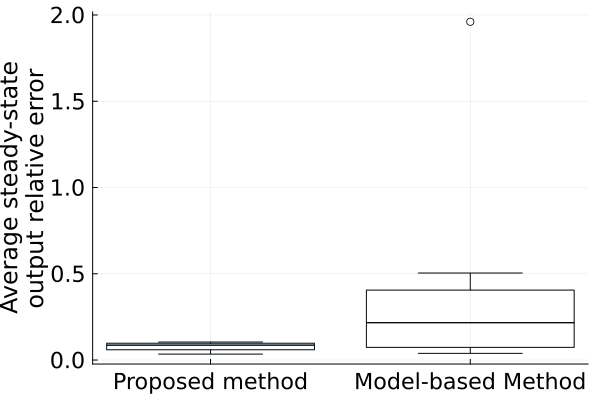} 
    \caption{Boxplots of the average steady-state output relative errors across 10 systems. }
    \label{fig:Boxplot_Average_FF}
\end{figure}
\begin{figure}[!t]
    \centering
    \includegraphics[width=0.65\columnwidth]{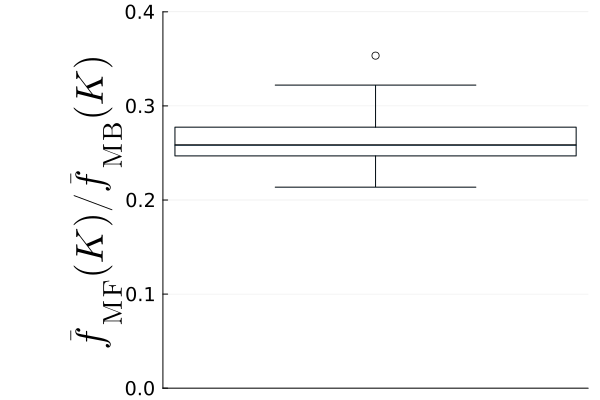} 
    \caption{Boxplot of the ratio of the PI gain performance metric $\frac{\bar{f}_\mathrm{MF}(K) }{ \bar{f}_\mathrm{MB}(K)}$ across 10 systems. }
    \label{fig:Boxplot_Average_PIGain}
\end{figure}

\section{Conclusion}\label{sec:concluding_remarks}
This paper has proposed a model-free design method for a 2DOF PI controller for MIMO LTI systems and has established its theoretical analysis. We have also provided theoretical and numerical comparison with the model-based method.
For feedforward design, we have tackled the system subject to Gaussian noise, which has been ignored in previous works, and have established analysis on the required time horizon and the feedforward estimation error.
For PI gain tuning, we have formulated an optimization problem and have analyzed sample complexity of a zeroth-order method.
Furthermore, we have quantified how the estimation inaccuracy of the feedforward affects the design of the PI controller.
Compared with the model-based method, the proposed method exhibits more favorable sample complexity when moderate feedforward accuracy is sufficient, although the model-based method is required for the high-accuracy feedforward.
In PI gain tuning, the proposed method achieves lower computational burden despite larger sample requirement.
Finally, numerical experiments have demonstrated that the proposed method has achieved superior control performance.\par
Future work includes developing a feedforward design method for nonlinear systems via the Koopman lifting technique \cite{Brunton_2021,Otto_Clarence_2021},
and improving the PI gain tuning method to reduce the sample complexity.
In addition, we conjecture that the iterates generated by Algorithm~\ref{algo:projected_grad_gain} converge to a local minimum.
This conjecture is motivated by a previous work \cite{Zhaolin_Yujie_na_2023} establishing that zeroth-order perturbed gradient descent with two-point estimators escapes strict saddle points, 
together with the observation that projected gradient descent behaves like vanilla gradient descent in the interior of the constraint $\Omega$, and that inexact cost evaluations can be interpreted as implicit perturbations.
A rigorous analysis is left for future work.
\section*{Acknowledgment}
This work was supported by the Japan Society for the Promotion of Science KAKENHI under Grant 23K28369.

\appendix

\subsection{Non-convexity of the cost function $f(K)$ }\label{appendix:Non-convexity}
We emphasize that the cost function $f(K)$ can be non-convex. Let $A=0.1$, $ B=1$, $C=1$, $W=0.5,$ and $V=Q_1=Q_2=1$. 
In this case, the cost function becomes $f(K)=\frac{(K_I+1)(0.5K_I+K_P^2K_I+1) }{K_I(K_P-0.1)} +\frac{K_P-0.1}{2K_I} +K_P$.
When we set $K^1=\begin{pmatrix}1 & 4\end{pmatrix}$ and $K^2=\begin{pmatrix}4 & 1.6\end{pmatrix}$, we have $\frac{f(K^1)+f(K^2)}{2}=12.49<13.55=f(\frac{K^1+K^2}{2})$, which indicates that $f(K)$ is non-convex.

\subsection{Well-definedness of Algorithm~\ref{algo:DataMatrix}}\label{appendix:well-definedness_FF}
The following lemma ensures that $E^{+}$ is well-defined with probability $1$, and consequently $\hat{u}_0$ is also well-defined.
\begin{lemma}\label{lem:fullrank_E}
    The matrix $E$ computed in (\ref{eq:DataMatrix_tau}) is full row rank with probability $1$.
\end{lemma}
\begin{proof}
    Each column of $E$ is a Gaussian random vector, which means that the support of $E$ is $\mathbb{R}^{n\times p}$.
    Since the set $\{E \in\mathbb{R}^{n\times p}\,|\, \text{The rows of $E$ are linearly dependent} \}$
    is a measure-zero subset of $\mathbb{R}^{n\times p}$,
    $E$ is full row rank with probability $1$.
\end{proof}

\subsection{Proof of Theorem~\ref{thm:size_of_tau_u}} \label{appendix:proof_of_SizeTauU}
Before presenting the proof of Theorem~\ref{thm:size_of_tau_u}, 
we introduce the covariance matrix of $e_x(t)$, which is employed in later, as
\begin{equation}\label{eq:var_Pcont}
\begin{split}
   &\Sigma_t\coloneqq \text{Var}\left(  e_x(t)\right)
    = \exp{(A_K t)} \Sigma_0 \exp{(A_K^\top t)} \\
   &\quad +\int_0^t  \exp{(A_K t)}
   ( W + B K_P V (B K_P)^\top )
   \exp{(A_K^\top t)} d\tau.
\end{split}
\end{equation} \par
We first establish the auxiliary lemma.

\begin{lemma}\label{lem:E-Estar}
    Define $E$ as (\ref{eq:DataMatrix_tau}).
    Set $\tau_u\geq \|Z\|_2\log M_2$.
    For any $\delta_u>0$, the following inequality holds with probability greater than $1-\delta_u$:
    \begin{equation}\label{eq:bound_E-Estar}
    \begin{split}
        &\| E -E^\star \|_\mathrm{F}\leq S_m(\delta_u)\\   
        &+\sqrt{mp} \| C\|_2 \|A_K^{-1}\|_2 \|B\|_2 \frac{\|Z\|_2}{\lambda_{\min}(Z)} \exp{\left( - \frac{\tau_u}{2\|Z \|_2}\right)}
        \end{split}
    \end{equation}
    where $Z,\,\mathfrak{G},\,c$ are defined in Theorem~\ref{thm:size_of_tau_u},
    and $S_m(\delta_u)$ is defined in (\ref{eq:Sm_FF}).
\end{lemma} 
\begin{proof} 

The deviation of $E$ from $E^\star$ can be bounded via the triangle inequality as,
    \begin{equation}\label{eq:bound_of_E-Estar}
    \|E -E^\star \|_\mathrm{F} \leq \| E- \mathbb{E}\left[ E \right] \|_\mathrm{F}  + \| \mathbb{E}\left[ E \right] -E^\star \|_\mathrm{F}. 
\end{equation}
We first bound the second term.
From (\ref{eq:mean_Pcont}), we obtain
$\mathbb{E} \left[ e^{\prime i}(t) \right]= -C \exp{(A_K t)} \left( A_K^{-1} B u_0^i \right) + C A_K^{-1}  Bu_0^i$,
and hence
\begin{equation}\label{eq:error_e_tau_prime}
    \mathbb{E} \left[ e^{\prime i}(\tau_u) - e^{\prime i}(\infty) \right]= -C  \exp{(A_K \tau_u)} \left( A_K^{-1} B u_0^i \right)
\end{equation}
holds.
Combining (\ref{eq:DataMatrix_tau}), (\ref{eq:Estar}), (\ref{eq:error_e_tau_prime}), and the fact that $\begin{pmatrix}u_0^1 & \dots &u_0^m\end{pmatrix} =I$, 
we have
\begin{equation*}
\begin{split}
     &\| \mathbb{E}\left[ E \right] -E^\star \|_\mathrm{F}
     =\|C \exp{(A_K \tau_u)}A_K^{-1}B\|_\mathrm{F}\\
     &\leq \sqrt{mp} \| C\|_2 \|A_K^{-1}\|_2 \|B\|_2 \| \exp(A_K \tau_u)\|_2\\
     &\leq \sqrt{mp} \| C\|_2 \|A_K^{-1}\|_2 \|B\|_2 \frac{\|Z\|_2}{\lambda_{\min}(Z)} \exp{\left( - \frac{\tau_u}{2\|Z \|_2}\right)},
     \end{split}
\end{equation*}
where the first inequality is due to $\|C \exp{(A_K \tau_u)}A_K^{-1}B\|_\mathrm{F}\leq
\sqrt{mp}\|C \exp{(A_K \tau_u)}A_K^{-1}B\|_2$,
and the last inequality follows from \cite[Lemma 12]{Convergence_Complexity}.\par

We now bound the first term in (\ref{eq:bound_E-Estar}).
Since $e_x(t)$ follows a normal distribution with covariance $\Sigma_t$,
each column of $ E- \mathbb{E}\left[ E \right]$ independently follows a normal distribution with mean $0$ and covariance $C \Sigma_{\tau_u} C^\top$.
Then we have $\mathbb{E}[\| E- \mathbb{E}\left[ E \right]\|_{\mathrm{F}}^2]=m\mathrm{tr}(C \Sigma_{\tau_u} C^\top)$ .
Applying \cite[Theorem 6.2.1]{Vershynin_2018} yields, for any $s>0$, 
\begin{equation*}
    \begin{split}
        &\mathbb{P}\left( 
        \left| \| E- \mathbb{E}\left[ E \right] \|_\mathrm{F}^2
        -m\mathrm{tr}\left(C \Sigma_{\tau_u} C^\top \right)\right|
        \geq s \right)\\
       & \leq 
        2\exp\left(-c \min \left\{
        \frac{s^2}{\mathfrak{G}^4 m \|C \Sigma_{\tau_u} C^\top \|_{\mathrm{F}}^2},
        \,
        \frac{s}{\mathfrak{G}^2 \|C \Sigma_{\tau_u} C^\top \|_2}
        \right\}
        \right).
    \end{split}
\end{equation*}
From the above concentration inequality and \cite[Lemma 11]{Soltanolkotabi_etal_2019}, we have
\begin{equation*}
\begin{split}
    &\quad\left\|  \| E- \mathbb{E}\left[ E \right] \|_\mathrm{F}^2
        -m\mathrm{tr}\left(C \Sigma_{\tau_u} C^\top \right) \right\|_{\psi_1}\\
      &  \leq \frac{9(\sqrt{2mc}+2)\mathfrak{G}^2 \|C \Sigma_{\tau_u} C^\top \|_{\mathrm{F}}}{2c}.
\end{split}
\end{equation*}
The above bound of sub-exponential norm implies that
\begin{equation}\label{eq:ProbBound_E}
    \begin{split}
        &\quad\mathbb{P}\left(\| E- \mathbb{E}\left[ E \right] \|_\mathrm{F}\leq 
        \sqrt{m\mathrm{tr}\left(C \Sigma_{\tau_u} C^\top \right)+ s }\right)\\
        &=\mathbb{P}(\| E- \mathbb{E}\left[ E \right] \|_\mathrm{F}^2\leq m\mathrm{tr}\left(C \Sigma_{\tau_u} C^\top \right)+ s )\\
        &\geq 1-
        2\exp\left(-c^2
        \frac{
        2s
        }{
        9(\sqrt{2mc}+2)\mathfrak{G}^2 \|C \Sigma_{\tau_u} C^\top \|_{\mathrm{F} }
        }
        \right)
    \end{split}
\end{equation} 
holds.\par

Lastly, we derive the upper bounds for $\mathrm{tr}\left(C \Sigma_{\tau_u} C^\top \right),\,\|C \Sigma_{\tau_u} C^\top \|_{\mathrm{F}}$.
Since $\Sigma_{\tau_u}$ is given by (\ref{eq:var_Pcont}), 
\begin{equation*}
    \begin{split}
       &\quad \mathrm{tr}\left(C \Sigma_{\tau_u} C^\top \right)\leq
        \|C\|_2^2 \|\exp(A_Kt)\|_2^2 \mathrm{tr}(\Sigma_0)+\mathrm{tr}(C\Sigma C^\top)\\
        &\leq  \frac{\|Z\|_2^2 \|C\|_2^2 \mathrm{tr}(\Sigma_0)}{\lambda_{\min}(Z)^2} \exp{\left( - \frac{\tau_u}{\|Z \|_2}\right)} + \mathrm{tr}(C\Sigma C^\top)
    \end{split}
\end{equation*}
holds, where $\Sigma$ is defined in (\ref{eq:Pcont_Lyapunov}).
If $\tau_u\geq \|Z\|_2 \log \frac{\|Z\|_2^2 \|C\|_2^2 \mathrm{tr}(\Sigma_0)}{\lambda_{\min}(Z)^2\mathrm{tr}(C\Sigma C^\top)} $, then we have $\mathrm{tr}\left(C \Sigma_{\tau_u} C^\top \right)\leq 2\mathrm{tr}(C\Sigma C^\top)$.
Similarly, if $\tau_u\geq \|Z\|_2 \log \frac{\|Z\|_2^2 \|C\|_2^2 \|\Sigma_0\|_\mathrm{F} }{\lambda_{\min}(Z)^2 \|C\Sigma C^\top\|_\mathrm{F} } $, then we obtain $\|C \Sigma_{\tau_u} C^\top \|_\mathrm{F} \leq 2\|C\Sigma C^\top\|_\mathrm{F}$.

\end{proof}

\begin{proof}[proof of Theorem~\ref{thm:size_of_tau_u}]
From (\ref{eq:def_of_hatu}) and (\ref{eq:ustar}), 
the following equality holds:
\begin{equation}\label{eq:bound_ustar_proof}
   \begin{split}
       &\quad \| \hat{u}_0-u^\star\|
= \|-E^+ e^0(\tau_u) +{E^\star}^+ { \mathbb{E}\left[ e^0(\infty)\right] } \| \\
&= \left\| \left({E^\star}^+ - {E}^+ \right) (e^0(\tau_u)-\mathbb{E} \left[ e^0(\infty)\right] )\right. \\
&\left. \quad - {E^\star}^+ (e^0(\tau_u)- \mathbb{E} \left[ e^0(\infty)\right] ) +
\left({E^\star}^+  - {E}^+ \right)  { \mathbb{E}\left[ e^0(\infty)\right]} \right\|\\
&{ \leq \|{E^\star}^+ -E^+\|_2
         \left(\|e^0(\tau_u)-\mathbb{E} \left[ e^0(\infty)\right] \|+\| \mathbb{E}\left[ e^0(\infty)\right] \|
         \right)}\\
        &{\quad +\frac{1}{\sigma_{p}(E^{\star})} \| ( e^0(\tau_u)-\mathbb{E} \left[ e^0(\infty)\right]) \|.}
   \end{split}
\end{equation} 
Using \cite[Theorem 4.1]{Wedin_1973}, we have
$ \|{E^\star}^+ -E^+\|_2 \leq 
    \sqrt{2}\frac{1}{\sigma_p(E)}   \| {E^\star}\|_2  \| {E^\star}-E\|_2$.

From $\sigma_p(E)\geq \sigma_p(E^\star)-\|E-E^\star\|_2$ \cite[Corollary 7.3.5]{horn2012matrix} and Lemma~\ref{lem:E-Estar}, we have
\begin{equation*}
    \begin{split}
       & \sigma_p(E)\geq \sigma_p(E^\star)-S_m(\delta_u)\\
       &\quad -\sqrt{mp} \| C\|_2 \|A_K^{-1}\|_2 \|B\|_2 \frac{\|Z\|_2}{\lambda_{\min}(Z)} \exp{\left( - \frac{\tau_u}{2\|Z \|_2}\right)}.
    \end{split}
\end{equation*}
If $\tau_u\geq 2\|Z\|_2\log\frac{4\sqrt{mp} \| C\|_2 \|A_K^{-1}\|_2 \|B\|_2 \|Z\|_2}{\lambda_{\min}(Z)\sigma_p(E^\star) }$, we have $\sigma_p(E)\geq\sigma_p(E^\star)/2$
with probability greater than
$1-2\exp\left(- 2c^2 
        \frac{
        \sigma_p(E^\star)^2/16-2m\mathrm{tr}(C\Sigma C^\top) 
        }{
        9(\sqrt{2mc}+2)\mathfrak{G}^2 \|C \Sigma C^\top \|_{\mathrm{F} }
        } \right)$.

This yields
\begin{equation}\label{eq:bound_of_psudoinv}
    \begin{split}
        \|{E^\star}^+ -E^+\|_2
    &\leq \frac{ 2{ \sqrt{2}}\|{E^\star}\|_2 \|{E^\star}-E\|_\mathrm{F} }{ \sigma_p (E^\star) },
    \end{split}
\end{equation}
with probability greater than $1-M_4$.

We next bound the quantity $\| ( e^0(\tau_u)-\mathbb{E} \left[ e^0(\infty)\right]) \|$.
By adding and subtracting $\mathbb{E} \left[ e^0(\tau_u)\right]$, we obtain
\begin{equation}\label{eq:bias_variance_e0}
\begin{split}
    &\| ( e^0(\tau_u)-\mathbb{E} \left[ e^0(\infty)\right]) \|\\
    &\leq 
    \| e^0(\tau_u)-\mathbb{E} \left[ e^0(\tau_u)\right]\|+
    \| \mathbb{E} \left[ e^0(\tau_u)\right] -\mathbb{E} \left[ e^0(\infty)\right] \|.
\end{split}
\end{equation}
From (\ref{eq:mean_Pcont}), we have
\begin{equation}\label{eq:bias_e0}
    \begin{split}
        &\quad\|\mathbb{E} \left[ e^0(\tau_u)\right] -\mathbb{E} \left[ e^0(\infty)\right] \|\\
        &=\|C \exp{(A_K\tau_u)} (\mathbb{E}[e_x(0)]-A_K^{-1}B u^\star )\|\\
        &\leq \|C\|_2 (D_x + \|A_K^{-1}\|_2 \|B\|_2 \|u^\star\| )\frac{\|Z\|_2}{\lambda_{\min}(Z)} \exp{\left( - \frac{\tau_u}{2\|Z \|_2}\right)}.
    \end{split}
\end{equation}
Using the same argument as the derivation of (\ref{eq:ProbBound_E}), for any $s>0$, 
the first term in (\ref{eq:bias_variance_e0}) is bounded as
\begin{equation}\label{eq:variance_e0}
    \begin{split}
       &\quad \mathbb{P}\left(  \| e^0(\tau_u)-\mathbb{E} \left[ e^0(\infty)\right] \|
    \leq \sqrt{ \mathrm{tr}(C \Sigma_{\tau_u} C^\top) +s}\right)\\
    &\geq 1-2\exp\left(-c^2         
        \frac{s}{9(\sqrt{2c}+2)\mathfrak{G}^2 \|C \Sigma C^\top \|_{\mathrm{F}}}
        \right).
    \end{split}
\end{equation}
Combining (\ref{eq:bound_ustar_proof}), (\ref{eq:bound_of_psudoinv}), (\ref{eq:bias_variance_e0}), (\ref{eq:bias_e0}) and (\ref{eq:variance_e0}), we obtain the result.

\end{proof}

\subsection{Proof of Theorem \ref{thm:Samp_comp_gain}}\label{appendix:proof_of_SampCompGain}
To prepare for the proof of Theorem~\ref{thm:Samp_comp_gain}, we provide several auxiliary lemmas.
The first lemma establishes a uniform stability margin for the closed-loop system over a sublevel set.
The second lemma ensures that $K\pm rU$ stabilizes the system if $K\in \mathcal{S}\cap \Omega$ and $r$ is sufficiently small.
Finally, we show that $f(K)$ is Hessian Lipschitz.

\begin{lemma}\label{lem:bound_of_smoothing}
    For any $a>0$ with $S(a)\neq \emptyset$, there exists $\mathfrak{f}_a>0$ such that $\mathrm{Re}(\lambda_{\max}(\bar{A}_K))\leq -\mathfrak{f}_a$, for all $K\in S(a)\cap\Omega$.
\end{lemma}
\begin{proof}[proof]
    
    We begin by proving the first statement.
    Under Assumption~\ref{assump:detectability}, it follows that $K\in \mathcal{S}$ if and only if $f(K)<\infty$.
    Therefore, for any $K\in S(a)$, the closed loop matrix $\bar{A}_K$ is Hurwitz.
    Since the set $S(a)\cap\Omega$ is compact and the eigenvalue $\lambda_{\max}(\bar{A}_K))$ is a continuous function of $K$, the statement holds. 
\end{proof}

The following lemma shows that a sufficiently small perturbation of $K\in \mathcal{S}\cap\Omega$ remains in
$\mathcal{S}$, where $\mathcal{S}$ is defined in Problem~\ref{problem:gain_optimization}. 
\begin{lemma}
    There exists $r_0> 0$ such that for any $0<r\leq r_0$, $K\in S(a)\cap \Omega$ and $U$ such that $\|U\|_\mathrm{F}=\sqrt{2mp}$, we have $K\pm rU \in \mathcal{S}$.
\end{lemma}
\begin{proof}
    Lemma~\ref{lem:bound_of_smoothing} implies that
    $\mathrm{Re}(\lambda_{\max}(\bar{A}_K))\leq -\mathfrak{f}_a$ holds for any $K\in S(a)\cap \Omega$.
    Because of the continuity of eigenvalues, there exists $r_0> 0$ and $\mathfrak{f}_a>\mathfrak{f}>0$ such that $\mathrm{Re}(\lambda_{\max}(\bar{A}_{K\pm rU}))\leq -\mathfrak{f}<0$, for any $0<r<r_0$, which completes the proof.
\end{proof}

The following lemma establishes that the cost function $f(K)$ is Hessian Lipschitz for all $K\in S(a)\cap\Omega$, which is crucial for the analysis on the smoothing parameter $r$.
\begin{lemma}\label{lem:H-Lipschitz}
We define a constant $H_a$ by
\begin{equation*}
    \frac{8\sqrt{n+p} a\mathfrak{s}\|\bar{B}\|_2^2 \|\bar{C}\|_2^2 }{\mathfrak{f}_a^2}
\left( \mathfrak{b}_a\mathfrak{f}_a 
        + \mathfrak{t}\|\bar{B}\|_2 
        +\|\bar{B}\|_2 \lambda_{\min}(V)\mathfrak{f}_a  \right),
\end{equation*}
where $\mathfrak{s}= \frac{1}{\lambda_{\min}(V) },\,
    \mathfrak{t}=(1+\mathfrak{d}_{\Omega} \|\bar{B}\|_2)\lambda_{\min}(V),\,
    \mathfrak{d}_{\Omega}=\max_{K\in\Omega}\|K\|_\mathrm{F} ,\,
        \mathfrak{c}= \lambda_{\min}(W)+\lambda_{\min}(V) (1+\|\bar{B}\|_2 \mathfrak{d}_{\Omega})^2$ and
    $\mathfrak{b}_a=\|\bar{B}\|_2 (\|\bar{C}\|_2 \mathfrak{c}/2\mathfrak{f}_a +\mathfrak{t})/\mathfrak{f}_a$.
Let $K^1,\,K^2$ be such that the line segment $K^1+t(K^2-K^1)$ with $t\in[0,1]$ lies in $ S(a) \cap \Omega$.
Then we have
\begin{equation*}
    \left| \nabla^2 f(K^1)[E,E] -\nabla^2 f(K^2)[E,E] \right|
    \leq H_a \|E\|_{\mathrm{F}}^2  \|K^1-K^2\|_{\mathrm{F}},
\end{equation*}
for all $E\in\mathbb{R}^{m\times 2p}$, where
$\nabla^2 f(K)[E,E]$ denotes the action on $E$ of the Hessian $\nabla^2 f$ at $K$.
\end{lemma} 
\begin{proof}[proof]
    Following the analysis in \cite[Lemma 3.11, 3.12]{Ilyas_Polyack2021}, we have
    \begin{align}
        &\nabla f(K) = -2 \bar{B}^\top Y X \bar{C}^\top 
        +
        2\bar{B}^\top Y  (\bar{I}_1+\bar{B}K \bar{I}_2)V \bar{I}_2^\top
        , \label{eq:grad_fK}\\
        &\frac{1}{2}\nabla^2 f(K)[E,E] = -2 \langle \bar{B}^\top Y X^\prime \bar{C}^\top,E\rangle \nonumber\\
        &+ \langle \bar{B}^\top Y^\prime (\bar{I}_1+\bar{B}K \bar{I}_2)V \bar{I}_2^\top ,E \rangle
        +\langle \bar{B}^\top Y\bar{B} E \bar{I}_2V \bar{I}_2^\top ,E \rangle, \nonumber
    \end{align}
    where $Y$, $X^\prime$ and $Y^\prime$ are the solutions of (\ref{eq:Lyap_Y}), (\ref{eq:X_prime}) and (\ref{eq:Y_prime}), respectively, and we define $\bar{I}_1=\begin{pmatrix}
        O\\I
    \end{pmatrix}$ and $\bar{I}_2=\begin{pmatrix}
        I\\O
    \end{pmatrix}$.
   Define $j(K)\coloneq\nabla^2 f(K)[E,E]$.
   Using the same argument as in the proof of \cite[Lemma 12]{Convergence_Complexity_IEEE}, 
   we obtain
   \begin{equation}\label{eq:perturb_Hessian}
    \begin{split}
        &\quad \nabla j(K)\\
        &=\bar{B}^\top 
        Y(-4W_1+2W_4+2W_5)\bar{C}^\top \\
        &\quad +\bar{B}^\top \{Y^\prime(-4X^\prime +2W_3) -4W_2 X
        \}\bar{C}^\top\\
        &\quad + \bar{B}^\top \{4W_2  (\bar{I}_1+\bar{B}K ) + 6Y^\prime\bar{B}E \}\bar{I}_2V\bar{I}_2^\top,
    \end{split}
    \end{equation}
    where $X^\prime ,\,Y^\prime,\,W_1,\,W_2,\,W_3,\,W_4$, and $W_5$ are defined in (\ref{eq:matrices_Lyapnov_H_Lipschitz}). \par
    \begin{figure*}[b]
    \centering
    \rule{\textwidth}{0.6pt}
    
\begin{align}
&\bar{A}_K^\top Y +Y \bar{A}_K +Q^\prime =O \label{eq:Lyap_Y}\\
    &\bar{A}_K X^\prime +X^\prime \bar{A}_K^\top - \bar{B} E \bar{C}X -( \bar{B} E \bar{C}X)^\top
           +\bar{B} E \bar{I}_2 V (\bar{I}_1+\bar{B}K \bar{I}_2)^\top
           +(\bar{I}_1+\bar{B}K \bar{I}_2) V (\bar{B} E \bar{I}_2)^\top
           =O \label{eq:X_prime}\\
    &\bar{A}_K^\top Y^\prime +Y^\prime \bar{A}_K
        - (\bar{B} E\bar{C})^\top Y - Y\bar{B} E\bar{C} =O \label{eq:Y_prime}
\end{align}
\begin{equation}\label{eq:matrices_Lyapnov_H_Lipschitz}
        \begin{split}
        &\bar{A}_K X^\prime +X^\prime \bar{A}_K^\top - \bar{B} E \bar{C}X -( \bar{B} E \bar{C}X)^\top
           +\bar{B} E \bar{I}_2 V (\bar{I}_1+\bar{B}K \bar{I}_2)^\top
           +(\bar{I}_1+\bar{B}K \bar{I}_2) V (\bar{B} E \bar{I}_2)^\top
           =O\\
            &\bar{A}_K W_1 +W_1 \bar{A}_K^\top -\bar{B}E\bar{C}X^\prime -(\bar{B}E\bar{C}X^\prime)^\top=O \\
           &  \bar{A}_K^\top W_2+W_2 \bar{A}_K -(\bar{B}E\bar{C})^\top Y^\prime -Y^\prime\bar{B}E\bar{C}=O\\
           &\bar{A}_K W_3+W_3 \bar{A}_K^\top -\bar{B}E \bar{I}_2 V (\bar{I}_1+\bar{B}K \bar{I}_2)^\top
           -(\bar{I}_1+\bar{B}K \bar{I}_2) V (\bar{B}E \bar{I}_2)^\top=O\\
          & \bar{A}_K W_4+W_4 \bar{A}_K^\top -W_3(\bar{B}E\bar{C})^\top -(\bar{B}E\bar{C}) W_3=O\\
          &\bar{A}_K W_5 + W_5 \bar{A}_K^\top -2\bar{B}E\bar{I}_2 V \bar{I}_2^\top E^\top \bar{B}^\top=O
        \end{split}
    \end{equation}
\end{figure*}
    To bound $\|\nabla j(K)\|_\mathrm{F}$, we derive bounds on the spectral norms of $X,\,Y,\, X^\prime ,\,Y^\prime,\,W_1,\ldots,W_5$.
    We begin with $Y$ and $X$.
    By applying the argument in \cite[Lemma C.2]{Ilyas_Polyack2021}, 
    \begin{equation*}\label{eq:bound_of_Y}
    \begin{split}
        f(K)-\mathrm{tr}(Q_1V)&=\mathrm{tr}(XQ^\prime)=
        \mathrm{tr}(Y\tilde{W}_K)\geq \mathrm{tr}(Y)\lambda_{\min}(\tilde{W}_K)\\
       & \geq \|Y\|_2 \lambda_{\min}(V)
         \end{split}
    \end{equation*}
    holds. Then we obtain $\|Y\|_2\leq a \mathfrak{s}$.
    Let $\mu$ denote the largest eigenvalue of $X$ and $v$ denote the corresponding eigenvector.
    Left- and right-multiplying $v^\top$ and $v$ to $\bar{A}_K X+ X \bar{A}_K^\top +\tilde{W}_K=O$ gives
    \begin{equation*}
    \begin{split}
        0&=v^\top (\bar{A}_K X+ X \bar{A}_K^\top +\tilde{W}_K)v =\mu v^\top (\bar{A}_K + \bar{A}_K^\top)v +v^\top \tilde{W}_Kv \\
        &\leq -2\mu \mathfrak{f}_a \|v\|^2 + \lambda_{\max}(\tilde{W}_K) \|v\|^2,
    \end{split}
    \end{equation*}
    and hence we have $\mu \leq \frac{ \lambda_{\max}(\tilde{W}_K) }{2\mathfrak{f}_a}$.
    Moreover, we bound $\lambda_{\max}(\tilde{W}_K)$ as
    $ \lambda_{\max}(\tilde{W}_K)= \lambda_{\max}(W)+ \lambda_{\max}\left(\begin{pmatrix}BK_P\\I\end{pmatrix} V \begin{pmatrix}BK_P\\I\end{pmatrix}^\top \right)
    \leq \lambda_{\max}(W)+\lambda_{\max}(V) (1+\|\bar{B}\|_2 \mathfrak{d}_{\Omega} )^2$,
    which implies that $\|X\|_2\leq \frac{\mathfrak{c} }{ 2\mathfrak{f}_a }$.
    By applying the same derivation as above,
    it holds that
    \begin{IEEEeqnarray*}{lCl}
            \|X^\prime \|_2 &&\leq  (\|\bar{B}\|_2 \|E\|_2 \|\bar{C}\|_2\|X\|_2 +\|\bar{B}\|_2 \|E \|_2\mathfrak{t}
            )/\mathfrak{f}_a
            \leq \mathfrak{b}_a\|E\|_2, \\
            \|Y^\prime \|_2&& \leq  \|\bar{B}\|_2 \|E\|_2 \|\bar{C}\|_2\|Y\|_2/\mathfrak{f}_a
            \leq a\mathfrak{s} \|\bar{B}\|_2 \|E\|_2 \|\bar{C}\|_2/\mathfrak{f}_a, \\
            \|W_1\|_2 &&\leq \|\bar{B}\|_2 \|E\|_2 \|\bar{C}\|_2 \|X^\prime\|_2/\mathfrak{f}_a
            \leq \mathfrak{b}_a \|\bar{B}\|_2  \|\bar{C}\|_2  \|E\|_2^2/2\mathfrak{f}_a^2, \\
            \|W_2\|_2&&\leq \|\bar{B}\|_2 \|E\|_2 \|\bar{C}\|_2 \|Y^\prime\|_2/\mathfrak{f}_a
            \leq a\mathfrak{s} \|\bar{B}\|_2^2 \|E\|_2^2 \|\bar{C}\|_2^2/\mathfrak{f}_a^2,\\
            \|W_3\|_2 &&\leq \|\bar{B}\|_2 \|E\|_2 \lambda_{\min}(V)(1+\|\bar{B}\|_2 \|K\|_2) /\mathfrak{f}_a \\
            &&=  \mathfrak{t} \|\bar{B}\|_2 \|E\|_2 / \mathfrak{f}_a,\\
            \|W_4\|_2&&\leq \|\bar{B}\|_2 \|E\|_2 \|W_3\|_2\|\bar{C}\|_2 / \mathfrak{f}_a
            \leq  \mathfrak{t}  \|\bar{B}\|_2^2 \|\bar{C}\|_2 \|E\|_2^2 / \mathfrak{f}_a^2,
            \\
            \|W_5\|_2&& \leq   \|\bar{B}\|_2^2 \|E\|_2^2 \lambda_{\min}(V)/\mathfrak{f}_a.
      \end{IEEEeqnarray*}
    Combining these bounds with (\ref{eq:perturb_Hessian}), we have $\|\nabla j(K)\|_\mathrm{F}\leq H_a \|E\|_\mathrm{F}^2$, which completes the proof.
    
\end{proof}

\begin{proof}[Proof of Theorem \ref{thm:Samp_comp_gain} ]
   The error between the estimated gradient $\hat{\nabla} f(K;\hat{u}_0)$ and the true gradient $\nabla f(K)$ can be decomposed as follows:
\begin{IEEEeqnarray}{lCl}
    &&\quad \left\| \nabla f(K)-\hat{\nabla} f(K;\hat{u}_0) \right\|_{\rm{F}} \nonumber \\
    &&\leq  \left\|\nabla f(K) -\frac{1}{N} \sum_{i=1}^N\langle\nabla f(K),\, U^i\rangle U^i \right\|_{\rm{F}} \nonumber\\
    &&\quad+ \left\|\frac{1}{N} \sum_{i=1}^N\langle\nabla f(K),\, U^i\rangle U^i-\bar{\nabla}f(K) \right\|_{\rm{F}}  \nonumber\\
   &&\quad +  \left\| \bar{\nabla}f(K)- \bar{\nabla}f_\tau (K)\right\|_{\rm{F}} + \left\| \bar{\nabla}f_\tau(K)- \bar{\nabla}f_\tau(K ;\hat{u}_0)\right\|_{\rm{F}}   \nonumber \\
    &&\quad  +\left\| \bar{\nabla}f_\tau(K; \hat{u}_0)- \hat{\nabla}f_\tau(K;\hat{u}_0)\right\|_{\rm{F}},  \label{eq:bound_of_estimator}
\end{IEEEeqnarray}
where
\begin{align}
       \bar{\nabla}f(K) &= \frac{1}{2rN} \sum_{i=1}^N (f(K+rU^i)-f(K-rU^i) ) U^i, \nonumber \\
         \bar{\nabla}f_{\tau}(K)& = \frac{1}{2rN} \sum_{i=1}^N (f_{\tau}(K+rU^i)-f_{\tau}(K-rU^i) ) U^i, \nonumber\\
         f_\tau(K)&=\mathbb{E}[ e^\star(\tau)^\top Q_1 e^\star(\tau) +z^\star(\tau)^\top Q_2 z^\star(\tau)  ], \nonumber\\
         \bar{\nabla}f_{\tau}(K;\hat{u}_0)&=\nonumber\\
         \frac{1}{2rN} &\sum_{i=1}^N (f_{\tau}(K+rU^i;\hat{u}_0)-f_{\tau}(K-rU^i ;\hat{u}_0) ) U^i, \label{eq:def_bar_nabla_uhat}\\
         f_\tau(K;\hat{u}_0)&=\mathbb{E}[ e(\tau)^\top Q_1 e(\tau) +z(\tau)^\top Q_2z(\tau) ] \label{eq:def_f_tau_uhat},
\end{align}
and $e^\star(t)$ and $z^\star (t)$ denote the trajectories of the error $e(t)$ and the integrator $z(t)$ under the feedforward $u_0=u^\star$.\par

We begin by bounding the first term in (\ref{eq:bound_of_estimator}).
Since $\mathbb{E}_U[ \langle\nabla f(K),\, U\rangle U] =\nabla f(K)$, we have
\begin{equation*}
    \begin{split}
        &\quad \left\|\nabla f(K) -\frac{1}{N} \sum_{i=1}^N\langle\nabla f(K),\, U^i\rangle U^i \right\|_{\rm{F}}\\
        &\leq 
        \left\|I -\frac{1}{N} \sum_{i=1}^N \mathrm{vec}(U^i)\mathrm{vec}(U^i)^\top\right\|_2
        \left\|\nabla f(K)\right\|_{\rm{F}}.
    \end{split}
\end{equation*}
From (\ref{eq:grad_fK}), we derive the upper bound of $\left\|\nabla f(K)\right\|_{\rm{F}}$, as follows
\begin{equation*}
\begin{split}
    \left\|\nabla f(K)\right\|_{\rm{F}} &\leq 2\|\bar{B}\|_2 \|\bar{C}\|_2 \|X\|_2 \|Y\|_{\mathrm{F}}\\
   &\quad +2 \|\bar{B}\|_2  \|Y\|_{\mathrm{F}} (1+\|\bar{B}\|_2 \|K\|_2) \lambda_{\min}(V)\\
    &\leq 2\sqrt{n+p}a \mathfrak{s}\mathfrak{b}_a \mathfrak{f}_a \eqqcolon \mathfrak{L}_a.
\end{split} 
\end{equation*}
Moreover, since $\left\|I -\mathrm{vec}(U_i)\mathrm{vec}(U_i)^\top\right\|_2=1$ and
$\left\|\sum_{i=1}^N \mathbb{E}[(I -\mathrm{vec}(U_i)\mathrm{vec}(U_i)^\top )^2] \right\|_2 = \left\|\sum_{i=1}^N (2mp-1)I  \right\|_2 =N(2mp-1)$,
the matrix Bernstein inequality \cite[Theorem 5.4.1]{Vershynin_2018} implies that, 
if
\begin{equation}\label{eq:size_OuterLoop}
    N\geq \frac{(12mp-6)\mathfrak{L}_a^2 +2\mathfrak{L}_a\epsilon/5}{3\epsilon^2/25}\log \frac{8mp}{\delta},
\end{equation}
then we have $\left\|\nabla f(K) -\frac{1}{N} \sum_{i=1}^N\langle\nabla f(K),\, U^i\rangle U^i \right\|_{\rm{F}}\leq \epsilon/5$ with probability greater than $1-\delta/2$.
\par

The second term in (\ref{eq:bound_of_estimator}) represents the bias induced by smoothing.
Following the proof of \cite[Proposition 5]{Convergence_Complexity_IEEE}, we have
$|\frac{1}{2r}(f(K+rU^i)-f(K-rU^i)) - \langle \nabla f(K),\, U^i\rangle | \leq \frac{r^2 (mp)^{3/2} H_{2a}}{2}$.
Consequently, it follows that
\begin{equation*}
    \left\|\frac{1}{N} \sum_{i=1}^N\langle\nabla f(K),\, U^i\rangle U^i-\bar{\nabla}f(K) \right\|_{\rm{F}}\leq \frac{r^2 (mp)^{2} H_{2a} }{2},
\end{equation*}
which provides the theoretical bound on smoothing parameter $r$ as:
\begin{equation}\label{eq:bound_smoothing}
    r\leq \frac{1}{mp} \sqrt{\frac{2\epsilon}{5H_{2a}} }.
\end{equation}
\par

The third term in (\ref{eq:bound_of_estimator}) arises from truncating the control horizon.
This term is bounded as:
\begin{equation*}
\begin{split}
    &\quad \left\| \bar{\nabla}f(K)- \bar{\nabla}f_\tau (K)\right\|_{\rm{F}}\\
    &=\| \frac{1}{2rN} \sum_{i=1}^N
    ( \{f(K+rU^i) -f_{\tau}(K+rU^i)\}  \\
   & \quad\quad  - \{f(K-rU^i) -f_{\tau}(K-rU^i) \} )U^i \|_\mathrm{F}\\
    &\leq \frac{\sqrt{2mp}  }{2rN } \sum_{i=1}^N (|f(K+rU^i) -f_{\tau}(K+rU^i) |\\
    &\quad\quad\quad\quad\quad\quad+|f(K-rU^i) -f_{\tau}(K-rU^i) |). \\
\end{split}
\end{equation*}
Since $ f_\tau(K) = \mathbb{E}[e^\star(\tau)^\top Q_1e^\star (\tau)+ z^\star(\tau)^\top Q_2 z^\star(\tau)]
    = \left\langle Q^\prime,\text{Var}\left( \begin{pmatrix} e_x^\star(\tau)\\ z^\star(\tau)\end{pmatrix} \right) \right\rangle +\langle Q_1,V\rangle$,
 we have
\begin{equation*}
    \begin{split}
        &\quad |f_\tau(K)-f(K)|= 
        \left| \left\langle Q^\prime,\text{Var}\left( \begin{pmatrix} e_x^\star(\tau)\\ z^\star(\tau)\end{pmatrix} \right) -X \right\rangle \right| \\
        &=\left| \left\langle Q^\prime, 
        \exp{(\bar{A}_K \tau)} \bar{\Sigma}_0 \exp{(\bar{A}_K^\top \tau)}\right\rangle \right|\\
       &\quad +\left| \left\langle Q^\prime, \int_\tau^\infty  \exp{(\bar{A}_K t)}
        \tilde{W}_K
        \exp{(\bar{A}_K^\top t)} dt
        \right\rangle \right| \\
    &\leq \lambda_{\max}(Q^\prime) \mathrm{tr}(\bar{\Sigma}_0) \|\exp(A_K \tau)\|^2_2\\
    &\quad +\lambda_{\max}(Q^\prime) \mathrm{tr}(\tilde{W}_K) \int_\tau^\infty \|\exp(A_K \tau)\|^2_2 dt,
    \end{split}
\end{equation*}
where the second equation follows from (\ref{eq:var_Pcont}) and 
$X=\int_0^\infty \exp(\bar{A}_K t) \tilde{W}_K \exp(\bar{A}_K^\top t)  dt $.
Applying \cite[Lemma 12]{Convergence_Complexity} yields
\begin{IEEEeqnarray*}{lCl}
    \|\exp(A_K t)\|^2_2 &&\leq (\|X\|_2 /\lambda_{\min}(X)) \exp{ (-(\lambda_{\min}(\tilde{W}_K)/\|X\|_2)t )}\\
    &&\leq (\|X\|_2 /\lambda_{\min}(X)) \exp{ (-t/\mathfrak{C} )},
\end{IEEEeqnarray*}
where $\mathfrak{C}= \frac{\mathfrak{c} }{ \mathfrak{f} \lambda_{\min}(V) } \leq \frac{ \|X\|_2}{ \|\tilde{W}_K \|_2 }$
and $\mathfrak{f}$ is defined in the proof of Lemma~\ref{lem:bound_of_smoothing}.
Furthermore, from \cite[Lemma A.5]{Ilyas_Polyack2021}, the smallest eigenvalue of $\lambda_{\min}(X)$ satisfies
$ \lambda_{\min}(X) \geq \frac{\lambda_{\min}(\tilde{W}_K)}{\|\bar{A}_K\|_2}
    \geq \frac{\lambda_{\min}(V)}{\mathfrak{e}_{\Omega}}$,
where we define
\begin{equation}\label{eq:bound_spec_AK}
    \mathfrak{e}_{\Omega} = \|\bar{A}\|_2+\|\bar{B}\|_2\|\bar{C}\|_2 \mathfrak{d}_{\Omega}\geq 
    \|\bar{A}_K\|_2.
\end{equation}
Therefore, 
\begin{equation}\label{eq:bound_of_expAK}
\begin{split}
    \|\exp(A_K t)\|^2_2 
    &\leq \frac{\mathfrak{c} \mathfrak{e}_{\Omega}}{ \mathfrak{f} \lambda_{\min}(V) }
    \exp{ \left(- \frac{t}{\mathfrak{C}} \right)}.
    \end{split}
\end{equation}
holds.
Combining these bounds,
we obtain
\begin{IEEEeqnarray*}{lCl}
    &\quad \left\| \bar{\nabla}f(K)- \bar{\nabla}f_\tau (K)\right\|_{\rm{F}}\leq \frac{\sqrt{2mp}  }{r } |f_\tau(K)-f(K)| \\
   &\leq \frac{\sqrt{2mp} \lambda_{\max}(Q^\prime) \mathfrak{c} \mathfrak{e}_{\Omega} }{r\mathfrak{f} \lambda_{\min}(V) } 
     \left(\mathrm{tr}(\bar{\Sigma}_0)+ \mathfrak{C}\mathrm{tr}(\tilde{W}_K)  \right)
     \exp{ (- \frac{\tau}{\mathfrak{C} })} \\
  &\leq \frac{\sqrt{2mp} \lambda_{\max}(Q^\prime) \mathfrak{c} \mathfrak{e}_{\Omega} }{ r\mathfrak{f} \lambda_{\min}(V) } 
     \left(\mathrm{tr}(\bar{\Sigma}_0)+ (n+p)\mathfrak{c} \mathfrak{C}\right)
     \exp{ (- \frac{\tau}{\mathfrak{C} })},
\end{IEEEeqnarray*}
where the last inequalities follows from $\mathrm{tr}(\tilde{W}_K)  \leq (n+p) \lambda_{\max}(\tilde{W}_K) \leq (n+p)\mathfrak{c} $.
The above inequality yields the lower bound on $\tau$.
\par

The fourth term in (\ref{eq:bound_of_estimator}) corresponds to the estimation error of the feedforward
$\hat{u}_0$, which is obtained in Algorithm~\ref{algo:DataMatrix}.
Using the triangle inequality, this term admits the decomposition
\begin{equation*}
\begin{split}
    &\quad \left\| \bar{\nabla}f_\tau(K)- \bar{\nabla}f_\tau (K; \hat{u}_0)\right\|_{\rm{F}}\\
    &=\| \frac{1}{2rN} \sum_{i=1}^N
    ( \{f_\tau(K+rU^i) -f_{\tau}(K+rU^i; \hat{u}_0 )\}  \\
   & \quad\quad  - \{f_\tau(K-rU^i) -f_{\tau}(K-rU^i ; \hat{u}_0) \} )U^i \|_\mathrm{F}\\
   &\leq \frac{\sqrt{2mp}  }{2rN } \sum_{i=1}^N \left( |f_\tau(K+rU^i) -f_{\tau}(K+rU^i; \hat{u}_0 ) | \right.\\
    &\left. \quad\quad\quad\quad\quad\quad+|f_\tau(K-rU^i) -f_{\tau}(K-rU^i; \hat{u}_0 ) | \right)
\end{split}
\end{equation*}
The definition of $f_\tau(K;\hat{u}_0)$ (\ref{eq:def_f_tau_uhat}) implies that
\begin{equation*}
    \begin{split}
        f_\tau(K;\hat{u}_0) &= \langle Q_1,V\rangle + \left\langle Q^\prime,\text{Var}\left( \begin{pmatrix} e_x(\tau)\\ z(\tau)\end{pmatrix} \right) \right\rangle \\
        &
        \quad+\left\langle Q^\prime,\mathbb{E}\left[\begin{pmatrix}e_x(\tau)\\z(\tau)\end{pmatrix} \right]
        \mathbb{E}\left[\begin{pmatrix}e_x(\tau)\\z(\tau)\end{pmatrix} \right]^\top
        \right\rangle,
    \end{split}
\end{equation*}
and hence we have
\begin{IEEEeqnarray*}{lCl}
        &&|f_\tau(K;\hat{u}_0)-f_\tau(K)| \\
        &&
        =\left| 
        \left\langle Q^\prime,
        \mathbb{E}\left[\begin{pmatrix}e_x(\tau)\\z(\tau)\end{pmatrix} \right]
        \mathbb{E}\left[\begin{pmatrix}e_x(\tau)\\z(\tau)\end{pmatrix} \right]^\top \right. \right.\\
      &&\left.\left. \quad\quad  -\mathbb{E}\left[\begin{pmatrix}e_x^\star(\tau)\\z^\star(\tau)\end{pmatrix} \right]
        \mathbb{E}\left[\begin{pmatrix}e_x^\star(\tau)\\z^\star(\tau)\end{pmatrix} \right]^\top
        \right\rangle
        \right|\\
        &&\leq \lambda_{\max}(Q^\prime) 
        \left( \left\|\mathbb{E}\left[\begin{pmatrix}e_x^\star(\tau)\\z^\star(\tau)\end{pmatrix} \right] \right\|
        +\left\| \mathbb{E}\left[\begin{pmatrix}e_x(\tau)\\z(\tau)\end{pmatrix} \right] \right\| \right)\\
        &&\quad\times \left\|\mathbb{E}\left[\begin{pmatrix}e_x^\star(\tau)\\z^\star(\tau)\end{pmatrix} \right] - \mathbb{E}\left[\begin{pmatrix}e_x(\tau)\\z(\tau)\end{pmatrix} \right] \right\|.
\end{IEEEeqnarray*}
From (\ref{eq:mean_var_augdynamics}), we obtain the bound
\begin{IEEEeqnarray*}{lCl}
        &&\left\|\mathbb{E}\left[\begin{pmatrix}e_x^\star(\tau)\\z^\star(\tau)\end{pmatrix} \right] - \mathbb{E}\left[\begin{pmatrix}e_x(\tau)\\z(\tau)\end{pmatrix} \right] \right\|\\
        &&\leq \left\{ \|\exp(\bar{A}_K \tau)\|_2 +1
        \right\}
        \|\bar{A}_K\|_2 \|\bar{B}\|_2 \gamma_u,\\
        &&\leq \left(
   \frac{\mathfrak{c} \mathfrak{e}_{\Omega}}{ \mathfrak{f} \lambda_{\min}(V) }
    \exp{ \left(- \frac{\tau}{2\mathfrak{C}} \right)} +1 
    \right) \mathfrak{e}_\Omega  \|\bar{B}\|_2 \gamma_u
\end{IEEEeqnarray*}
where the last inequality follows from (\ref{eq:bound_spec_AK}) and (\ref{eq:bound_of_expAK}),
and $\gamma_u\coloneq \|\hat{u}_0-u^\star\|$.
if $\tau$ is chosen sufficiently large so that
$   \frac{\mathfrak{c} \mathfrak{e}_{\Omega}}{ \mathfrak{f} \lambda_{\min}(V) }
    \exp{ \left(- \frac{\tau}{2\mathfrak{C}} \right)} \leq 1$,
then it holds that
\begin{IEEEeqnarray*}{lCl}
    \left\| \bar{\nabla}f_\tau(K)- \bar{\nabla}f_\tau(K;\hat{u}_0)\right\|_{\rm{F}} 
    =O(\gamma_u^2).
\end{IEEEeqnarray*} \par
Finally, we bound the last term of (\ref{eq:bound_of_estimator}), which arises from the 
variance of $\frac{1}{N_{\mathrm{sub}}}\sum_{j=1}^{N_{\mathrm{sub}}}\hat{f}^{i,j,k}$, computed in Step~7 of Algorithm~\ref{algo:estimation_grad_cost}.
The simulated cost $\hat{f}^{i,j,k}$ is equivalent to
$\left\langle Q^\prime, \begin{pmatrix} e_x(\tau)\\z(\tau)\end{pmatrix} \begin{pmatrix} e_x(\tau)\\z(\tau)\end{pmatrix}^\top \right\rangle+\langle Q_1,vv^\top \rangle$,
where $e_x(\tau)$ and $z(\tau)$ are the error state and the integrator 
in Step~5 of Algorithm~\ref{algo:estimation_grad_cost}, respectively.
This implies $f_\tau(K^{i,k} ;\hat{u}_0)=\mathbb{E}[ \hat{f}^{i,j,k}]$, where $K^{i,1}$ and $K^{i,2}$ are defined in Step~3 of Algorithm~\ref{algo:projected_grad_gain}.
The Hanson-Wright inequality \cite[Theorem 6.2.1]{Vershynin_2018} implies that for sufficiently small $\epsilon$,
we have
\begin{equation*}
\begin{split}
   &\quad \mathbb{P}\left\{ \left| \frac{1}{N_{\mathrm{sub}}}\sum_{j=1}^{N_{\mathrm{sub}}}\hat{f}^{i,j,k}
    -f_\tau(K^{i,k};\hat{u}_0)\right|\geq \epsilon
    \right\}\\
   & \leq 2\exp{\left( -c N_{\mathrm{sub}}
   \min\left\{
    \frac{ \epsilon^2}
    {\mathfrak{G}^4 \| S_{\tau} \|_{\mathrm{F}}^2 } ,
     \frac{ \epsilon}
    {\mathfrak{G}^2 \| S_{\tau} \|_2 }
    \right\}
    \right)} 
    \end{split}
\end{equation*}
where $c$ is a absolute positive constant, and
 $S_{\tau}\coloneq\bar{\Sigma}_\tau^{\frac{1}{2}} Q^\prime \bar{\Sigma}_\tau^{\frac{1}{2}}
    +V^{\frac{1}{2}}Q_1 V^{\frac{1}{2}}$. 
From the above concentration inequality and \cite[Lemma 11]{Soltanolkotabi_etal_2019}, we have
\begin{equation*}
\begin{split}
    &\quad \left\| \frac{1}{N_{\mathrm{sub}}}\sum_{j=1}^{N_{\mathrm{sub}}}\hat{f}^{i,j,k}
    -f_\tau(K^{i,k};\hat{u}_0)\right\|_{\psi_1}
    \leq \frac{9\mathfrak{G}^2 \|P_{\tau_u} \|_{\mathrm{F}}}{ 
      \sqrt{2cN_{\mathrm{sub}} }
      }.
\end{split}
\end{equation*}
The above bound of sub-exponential norm implies that 
\begin{equation}\label{eq:prob_ineq_innerloop}
\begin{split}
     &\quad \mathbb{P}\left\{ \left| \frac{1}{N_{\mathrm{sub}}}\sum_{j=1}^{N_{\mathrm{sub}}}\hat{f}^{i,j,k}
    -f_\tau(K^{i,k};\hat{u}_0)\right|\geq s
    \right\}\\
    &\leq
    2 \exp \left(
    -\frac{c \sqrt{2cN_{\mathrm{sub}}} }{9\mathfrak{G}^2 \|S_{\tau
}\|_{\mathrm{F}} } s
    \right).
\end{split}
\end{equation}
From (\ref{eq:nabla_hat}) and (\ref{eq:def_bar_nabla_uhat}),
the last term of (\ref{eq:bound_of_estimator}) is bounded as follows:
\begin{IEEEeqnarray*}{lCl}
       && \left\| \bar{\nabla}f_\tau(K; \hat{u}_0)- \hat{\nabla}f_\tau(K:\hat{u}_0)\right\|_{\rm{F}}\\
       && \leq \frac{\sqrt{2mp}}{2rN } \sum_{i=1}^N \left(
       \left|f_{\tau}(K^{i,1};\hat{u}_0)-\frac{1}{N_{\mathrm{sub}}}\sum_{j=1}^{N_{\mathrm{sub}}} \hat{f}^{i,j,1}\right| \right.\\
      &&\left. \quad\quad\quad\quad\quad\quad\quad +\left|f_{\tau}(K^{i,2};\hat{u}_0)-\frac{1}{N_{\mathrm{sub}}}\sum_{j=1}^{N_{\mathrm{sub}}} \hat{f}^{i,j,2}\right|
       \right).
    \end{IEEEeqnarray*}
Combining the above inequality and (\ref{eq:prob_ineq_innerloop}) completes the proof.

\end{proof}

\subsection{ The analysis on model-based method }\label{appendix:indirect_approach}

We establish perturbation bounds on the solutions to the Lyapunov equations.
These bounds are used to quantify the perturbation of the gradient $\nabla f_\mathrm{d}(K)-\nabla f_\mathrm{Id}(K)$ in the following theorem.
\begin{lemma}\label{lem:perturb_X}
Suppose that (\ref{eq:assumption_identification}) and (\ref{eq:condition_perturb_X}) hold,
where $\mathfrak{d}_\Omega$ in (\ref{eq:condition_perturb_X}) is defined in Lemma~\ref{lem:bound_of_smoothing}.
Let $X_\mathrm{d}$ and $X_\mathrm{Id}$ be the solutions to the following Lyapunov Equations
        $\bar{A}_{K,\mathrm{d} } X_\mathrm{d} \bar{A}_{K,\mathrm{d} }^\top -X_\mathrm{d} 
        + \begin{pmatrix} W_d + B_\mathrm{d} K_P V (B_\mathrm{d} K_P)^\top & B_\mathrm{d} K_P V \\ (B_\mathrm{d} K_P V)^\top& V\end{pmatrix} =O$
and 
$\bar{A}_{K,\mathrm{Id}} X_{\mathrm{Id}} \bar{A}_{K,\mathrm{Id}}^\top - X_{\mathrm{Id}} 
        +\begin{pmatrix} W_\mathrm{Id} + B_{\mathrm{Id}} K_P V_{\mathrm{Id}} (B_{\mathrm{Id}} K_P)^\top & B_{\mathrm{Id}} K_P V_{\mathrm{Id}} \\ (B_{\mathrm{Id}} K_P V_{\mathrm{Id}})^\top& V_{\mathrm{Id}}\end{pmatrix} =O$,
respectively. Similarly, let $Y_\mathrm{d}$
and $Y_\mathrm{Id}$ be the solutions to the following Lyapunov Equations
$\bar{A}_{K,\mathrm{Id} }^\top Y_\mathrm{Id}  \bar{A}_{K,\mathrm{Id}} -Y_\mathrm{Id} +Q^\prime_\mathrm{Id}=O$
and
$\bar{A}_{K,\mathrm{d} }^\top Y_\mathrm{d}  \bar{A}_{K,\mathrm{d}} -Y_\mathrm{d} +Q^\prime_\mathrm{d}=O$, respectively.
Then we have $\|X_\mathrm{d}  -X_{\mathrm{Id}}\|_{\mathrm{F}},\, \|Y_\mathrm{d}  -Y_{\mathrm{Id}}\|_{\mathrm{F}}=O(\epsilon_\mathrm{Id})$.
\end{lemma}
\begin{figure*}[b]
    \centering
    \rule{\textwidth}{0.6pt}
    \textcolor{black}{
\begin{equation}\label{eq:condition_perturb_X}
\begin{split}
   & \epsilon_\mathrm{Id}\leq
    \frac{-(\|A_\mathrm{d}\|_2+\|C_\mathrm{d}\|_2 +1)+\sqrt{(\|A_\mathrm{d}\|_2+\|C_\mathrm{d}\|_2 +1)^2+4 (\mathfrak{d}_\Omega +h)\mathfrak{u}}}{2 (\mathfrak{d}_\Omega +1) }\\
    &\mathfrak{u}=-\|\bar{A}_{K,\mathrm{d} }\|_2 +\sqrt{ \|\bar{A}_{K,\mathrm{d} }\|_2^2 +\frac{1}{2\|P^{-1}\|_2}}
    \quad
    P=I\otimes I -\bar{A}_{K,\mathrm{d}}^\top \otimes \bar{A}_{K,\mathrm{d}}^\top
\end{split}
\end{equation}    
    }
\end{figure*}
\begin{proof}
    Applying \cite[Theorem 2.6]{Gahinet_etal_1990} yields
    \begin{equation*}
        \begin{split}
            \|X_\mathrm{d} -X_\mathrm{Id} \|_2 &\leq \|P^{-1}\|_2 \left\| \Delta \tilde{W}_K\right\|_2 \\
            &\quad+\{\| P^{-1}\|_2(2\|\bar{A}_{K,\mathrm{d} }\||_2 +\|\bar{A}_{K,\mathrm{d} }-\bar{A}_{K,\mathrm{Id} }\||_2 \}\\
            &\quad\quad\times(\|X_\mathrm{d} \|_2 +\|X_\mathrm{d} -X_\mathrm{Id} \|_2 ),
        \end{split}
    \end{equation*}
    where $\Delta \tilde{W}_K$ denotes
    $\begin{pmatrix} W_d + B_\mathrm{d} K_P V (B_\mathrm{d} K_P)^\top & B_\mathrm{d} K_P V \\ (B_\mathrm{d} K_P V)^\top& V\end{pmatrix}-
        \begin{pmatrix} W_\mathrm{Id} + B_{\mathrm{Id}} K_P V_{\mathrm{Id}} (B_{\mathrm{Id}} K_P)^\top & B_{\mathrm{Id}} K_P V_{\mathrm{Id}} \\ (B_{\mathrm{Id}} K_P V_{\mathrm{Id}})^\top& V_{\mathrm{Id}}\end{pmatrix}$.
    Rearranging gives
    \begin{equation}\label{eq:_bound_of_purturbX}
        \begin{split}
           & \left(
            1- \|P^{-1}\|_2 
             \{2\|\bar{A}_{K,\mathrm{d} }\|_2 +\ell(\epsilon_\mathrm{Id} ) \}
             \ell(\epsilon_\mathrm{Id})
            \right)\|X_\mathrm{d} -X_\mathrm{Id} \|_2\\
           & \leq 
             \|P^{-1}\|_2 \left\| \Delta \tilde{W}_K\right\|_2 + \|P^{-1}\|_2 
             \{2\|\bar{A}_{K,\mathrm{d} }\|_2 +\ell(\epsilon_\mathrm{Id} ) \}
             \ell(\epsilon_\mathrm{Id})
        \end{split}
    \end{equation}
    where we define $\ell(\epsilon_\mathrm{Id})=\|\bar{A}_{K,\mathrm{d} }-\bar{A}_{K,\mathrm{Id} }\|_2$.
    The perturbation of the closed-loop matrix $\ell(\epsilon_\mathrm{Id})$ is bounded as follows:
    \begin{equation*}
        \begin{split}
            &\quad \ell(\epsilon_\mathrm{Id})\\
            &\leq \left\|\begin{pmatrix}
                A_\mathrm{d}- A_\mathrm{Id}& O\\ -(CA_\mathrm{d}- C_\mathrm{Id} A_\mathrm{Id}) &O
            \end{pmatrix} \right\|_2\\
            &\quad+ \left\|\begin{pmatrix}
                B_\mathrm{d}- B_\mathrm{Id}\\ O
            \end{pmatrix} \right\|_2 
            \|K\|_2
            \left\|\begin{pmatrix}
                C_\mathrm{d}- C_\mathrm{Id}& O\\ O &O
            \end{pmatrix} \right\|_2\\
            &\leq \| A_\mathrm{d}- A_\mathrm{Id}\|_2 + \|CA_\mathrm{d}- C_\mathrm{Id} A_\mathrm{Id}\|_2\\
           &\quad +\| B_\mathrm{d}- B_\mathrm{Id}\|_2 \|K\|_2 \| C_\mathrm{d}- C_\mathrm{Id}\|_2\\
           &\leq \epsilon_\mathrm{Id} +(\epsilon_\mathrm{Id} \|A_\mathrm{d}\|_2 + (\|C_\mathrm{d}\|_2 +\epsilon_\mathrm{Id})\epsilon_\mathrm{Id})+ \mathfrak{d}_\Omega \epsilon_\mathrm{Id}^2\\
            &= (1+\mathfrak{d}_\Omega)\epsilon_\mathrm{Id}^2 
            +(\|A_\mathrm{d}\|_2 + \|C_\mathrm{d}\|_2+1) \epsilon_\mathrm{Id}
        \end{split}
    \end{equation*}
    Suppose the condition (\ref{eq:condition_perturb_X}) holds. Then we have
    \begin{equation*}
    \begin{split}
       \ell(\epsilon_\mathrm{Id})\leq
        -\|\bar{A}_{K,\mathrm{d} }\|_2 +\sqrt{\|\bar{A}_{K,\mathrm{d} }\|_2^2 +1/2\|P^{-1}\|_2},
    \end{split}
    \end{equation*}
    and subsequently 
    \begin{equation*}
         \|P^{-1}\|_2 
             \{2\|\bar{A}_{K,\mathrm{d} }\|_2 +\ell(\epsilon_\mathrm{Id} ) \}
             \ell(\epsilon_\mathrm{Id})\leq \frac{1}{2}
    \end{equation*}
    holds.
    Substituting this inequality into (\ref{eq:_bound_of_purturbX}) shows that
    \begin{equation*}
        \begin{split}
           & \|X_\mathrm{d} -X_\mathrm{Id} \|_2 \leq 
            2 \|P^{-1}\|_2 \left\| \Delta \tilde{W}_K\right\|_2 +1.
        \end{split}
    \end{equation*}
    From the above inequality
    and $\|\Delta \tilde{W}_K\|_2 =O(\epsilon_\mathrm{Id})$,
    which completes the proof for $\|X_\mathrm{d} -X_\mathrm{Id} \|_\mathrm{F}$.
    A similar argument yields $\|Y_\mathrm{d} -Y_\mathrm{Id} \|_\mathrm{F}=O(\epsilon_\mathrm{Id})$.
\end{proof}

The following theorem shows that we can obtain an $O(\epsilon_\mathrm{Id})$-stationary point of $f_\mathrm{d}(K)$ via applying a projected gradient descent to (\ref{prob:PI_Identified}).
\begin{theorem}\label{thm:degradation_indirect}
    Suppose that the inequality (\ref{eq:assumption_identification}) and the condition (\ref{eq:condition_perturb_X}) in Lemma~\ref{lem:perturb_X} hold. 
    Let $K^\star_{\mathrm{Id}}$ be $\epsilon_\mathrm{Id}$-stationary point of $f_\mathrm{Id}(K)$ with the step size $\eta$, where an
    $\epsilon$-stationary point is defined in Definition~\ref{def:stationary_point}.
    Then $K^\star_{\mathrm{Id}}$ is an $O(\epsilon_\mathrm{Id})$-stationary point of $f_\mathrm{d}(K)$.
\end{theorem}
\begin{proof}
    The criterion $\left\|
          \frac{\mathrm{proj}_\Omega (K_\mathrm{Id}^\star - \eta\nabla f_\mathrm{d}(K_\mathrm{Id}^\star))
          -K_\mathrm{Id}^\star}{\eta} \right\|_\mathrm{F}$ is bounded as follows
    \begin{IEEEeqnarray*}{lCl}
           && \left\|
          \frac{\mathrm{proj}_\Omega (K_\mathrm{Id}^\star - \eta\nabla f_\mathrm{d}(K_\mathrm{Id}^\star))
          -K_\mathrm{Id}^\star}{\eta} \right\|_\mathrm{F} \leq
          \\
         & &
           \frac{1}{\eta} \left\|
          \mathrm{proj}_\Omega (K_\mathrm{Id}^\star - \eta\nabla f_\mathrm{d}(K_\mathrm{Id}^\star))
          -\mathrm{proj}_\Omega (K_\mathrm{Id}^\star - \eta\nabla f_\mathrm{Id}(K_\mathrm{Id}^\star))
          \right\|_\mathrm{F} \\
         && + \frac{1}{\eta}\left\|
          \mathrm{proj}_\Omega (K_\mathrm{Id}^\star - \eta\nabla f_\mathrm{Id}(K_\mathrm{Id}^\star))
          -K_\mathrm{Id}^\star
          \right\|_\mathrm{F} \\
        &&\leq  \left\|
          \nabla f_\mathrm{d}(K_\mathrm{Id}^\star))
          -\nabla f_\mathrm{Id}(K_\mathrm{Id}^\star)
          \right\|_\mathrm{F} 
          + \frac{\epsilon_\mathrm{Id} }{\eta},
        \end{IEEEeqnarray*}
    where the last inequality holds because projections onto convex sets are contractive.
    Since the gradient expressions of $\nabla f_\mathrm{Id}(K)$ and $\nabla f_\mathrm{d}(K)$ are given by (\ref{eq:gradient_discrete_id}) and its true parameters counterparts, respectively,
    we obtain
    $ \left\|\nabla f_\mathrm{d}(K_\mathrm{Id}^\star))-\nabla f_\mathrm{Id}(K_\mathrm{Id}^\star))\right\|_\mathrm{F}\\
            =O(\max\{ \epsilon_\mathrm{Id},\,\|X_\mathrm{d}  -X_{\mathrm{Id}}\|_{\mathrm{F}},\,
            \|Y_\mathrm{d}  -Y_{\mathrm{Id}}\|_{\mathrm{F}} \} )$.
    Lemma~\ref{lem:perturb_X} provides that $\|Y_\mathrm{d}  -Y_{\mathrm{Id}}\|_{\mathrm{F}}, \|X_\mathrm{d}  -X_{\mathrm{Id}}\|_{\mathrm{F}} =O(\epsilon_\mathrm{Id})$, which completes the proof.
\end{proof}

\bibliographystyle{IEEEtran}
\bibliography{main.bib}

\end{document}